\documentclass{article}
\usepackage{bm}
\usepackage{tikz,mathpazo}
\usepackage{pgf}
\usepackage[top=2.5cm, bottom=2.5cm, left=3cm, right=3cm]{geometry}   
\usepackage{indentfirst}
\usepackage{graphicx}
\usepackage{graphics}
\usepackage[toc,page,title,titletoc,header]{appendix}
\usepackage{bm}
\usepackage{bbm}
\usepackage{listings}  
\usepackage{amsmath}
\usepackage{setspace} 
\usepackage{indentfirst}
\usepackage{caption}
\usepackage{multirow} 
\usepackage{lipsum,multicol}
\usepackage{pdfpages}
\usepackage{float}
\usepackage{amsthm}
\usepackage{pdfpages}
\usepackage{url}
\usepackage{colortbl}
\usepackage{subfigure}
\usepackage{epsfig}
\usepackage{epstopdf}
\usetikzlibrary{shapes.geometric, arrows}
\usepackage{fancyhdr}
\usepackage{abstract}
\usepackage{tikz,mathpazo}
\usetikzlibrary{shapes.geometric, arrows}
\usepackage{amssymb}
\usepackage{latexsym}
\usepackage{verbatim}
\usepackage[numbers]{natbib} 
\usepackage{booktabs}

\usepackage{hyperref}
\hypersetup{
	colorlinks=true,
	linkcolor=blue,
	filecolor=magenta,      
	urlcolor=cyan,
	citecolor=blue,
}

\allowdisplaybreaks[4]

\newcommand{\inner}[1]{\left\langle #1 \right\rangle}
\newcommand{\norm}[1]{\left\Vert #1\right\Vert}
\newcommand{\bb}[1]{\mathbb{#1}}

\newcommand{\conv}[0]{\mathrm{conv}\,}

\newcommand{\ca}[1]{\mathcal{#1}}

\newcommand{\mk}{{m_{k} }}
\newcommand{\vk}{{v_{k} }}
\newcommand{\mkp}{{m_{k+1} }}
\newcommand{\vkp}{{v_{k+1} }}

\newcommand{\xk}{{x_{k} }}
\newcommand{\yk}{{y_{k} }}

\newcommand{\xkp}{{x_{k+1} }}

\newcommand{\ykp}{{y_{k+1} }}

\newcommand{\D}{\ca{D}}

\newcommand{\Rn}{\mathbb{R}^n}

\newtheorem{theo}{Theorem}[section]
\newtheorem{lem}[theo]{Lemma}
\newtheorem{prop}[theo]{Proposition}

\newtheorem{coro}[theo]{Corollary}

\newtheorem{defin}[theo]{Definition}
\newtheorem{rmk}[theo]{Remark}
\newtheorem{assumpt}[theo]{Assumption}

\usepackage[figuresright]{rotating}
\usepackage{pdflscape}

\usepackage{caption}
\usepackage{algorithm}
\usepackage{algpseudocode}
\usepackage{longtable}
\usepackage{appendix}
\usepackage{authblk}

\numberwithin{equation}{section}

\usepackage{array}

\title{Adam-family Methods with Decoupled Weight Decay in Deep Learning\thanks{The first two authors have made equal contributions to this paper.}}
\author{Kuangyu Ding\thanks{Department of Mathematics, National University of Singapore, Singapore (kuangyud@u.nus.edu).},~
	Nachuan Xiao\thanks{Institute of Operations Research and Analytics, National University of Singapore, Singapore (xnc@lsec.cc.ac.cn).},~
	Kim-Chuan Toh\thanks{Department of Mathematics, and Institute of Operations Research and Analytics, National University of Singapore, Singapore 119076 (mattohkc@nus.edu.sg).}}

\begin{document}
	\maketitle
	
	\begin{abstract}
		In this paper, we investigate the convergence properties of a wide class of Adam-family methods for minimizing quadratically regularized nonsmooth nonconvex optimization problems, especially in the context of training nonsmooth neural networks with weight decay. Motivated by the AdamW method, we propose a novel framework for Adam-family methods with decoupled weight decay. Within our framework, the estimators for the first-order and second-order moments of stochastic subgradients are updated independently of the weight decay term. Under mild assumptions and with non-diminishing stepsizes for updating the primary optimization variables, we establish the convergence properties of our proposed framework. In addition, we show that our proposed framework encompasses a wide variety of well-known Adam-family methods, hence offering convergence guarantees for these methods in the training of nonsmooth neural networks. More importantly, we show that our proposed framework asymptotically approximates the SGD method, thereby providing an explanation for the empirical observation that decoupled weight decay enhances generalization performance for Adam-family methods. As a practical application of our proposed framework, we propose a novel Adam-family method named  Adam with Decoupled Weight Decay (AdamD), and establish its convergence properties under mild conditions. Numerical experiments demonstrate that AdamD outperforms Adam and is comparable to AdamW, in the aspects of both generalization performance and efficiency.
	\end{abstract}
	
\section{Introduction}
We consider the following unconstrained stochastic optimization problem:
\begin{equation}
	\label{Prob_Ori}
	\tag{UOP}
	\begin{aligned}
		\min_{x \in \Rn}\quad  g(x) := f(x) + \frac{\sigma}{2} \norm{x}^2,
	\end{aligned}
\end{equation}
where the function $f: \Rn  \to \bb{R}$ is assumed to be locally Lipschitz continuous and possibly nonsmooth over $\Rn$. Moreover, the constant $\sigma > 0$ is the penalty parameter for the quadratic regularization term. Such a  regularization term is also known as the weight decay term, which is widely employed to enhance the generalization performance in training neural networks \cite{bos1996using,krogh1991simple}. 

The stochastic gradient descent (SGD) is one of the most fundamental methods for solving \eqref{Prob_Ori}. In the SGD method, all coordinates of the variable $x$ are updated with the same stepsize (i.e., learning rate). To accelerate the SGD method, the widely used Adam method \cite{kingma2014adam} is developed by adjusting the coordinate-wise stepsizes based on first-order and second-order moments of the stochastic gradients. Due to its high efficiency in training neural networks, the Adam method has become one of the most popular choices for various neural network training tasks.

Motivated by the Adam method, numerous efficient Adam-family methods {have been} developed, such as  AdaBelief \cite{zhuang2020adabelief}, AMSGrad \cite{reddi2019convergence}, Yogi \cite{zaheer2018adaptive}, etc. From a theoretical perspective, the majority of existing works \cite{barakat2021convergence,guo2021novel,shi2021rmsprop,wang2022provable,zaheer2018adaptive,zhang2022adam,zou2019sufficient} establish convergence properties for these Adam-family methods, based on the assumption that $f$ is continuously differentiable over $\Rn$.  However, as emphasized  in \cite{bolte2021nonsmooth,bolte2021conservative,bolte2022differentiating}, nonsmooth activation functions, including ReLU and leaky ReLU, are popular choices in building neural networks. For any neural network built from these nonsmooth activation functions, its loss function is usually nonsmooth and lacks Clarke regularity (e.g., differentiability, weak convexity, etc.). Consequently, these existing works are unable to provide convergence guarantees for their analyzed methods in the training of nonsmooth neural networks.

\paragraph{\bf Existing works on training nonsmooth neural networks}

In nonsmooth optimization, it has been demonstrated in \cite{daniilidis2020pathological} that a general Lipschitz continuous functions $f$ can exhibit highly pathological properties, leading to the failure of subgradient descent method to find any critical point of $f$. Moreover, the chain rule may fail for the Clarke subdifferential \cite{clarke1990optimization} of the loss function of a nonsmooth neural network. Specifically, when we differentiate the loss function of a nonsmooth neural network using automatic differentiation (AD) algorithms, the outputs may not be contained in the Clarke subdifferential of $f$ \cite{bolte2020mathematical}.

Consequently, most of the existing works restrict their analysis to the class of {\it path-differentiable} functions \cite[Definition 3]{bolte2021conservative}.  For any path-differentiable function $f$, there exists a graph-closed set-valued mapping $\D_f$, called {\it conservative field} for $f$, such that for any absolutely continuous mapping $\gamma: [0, \infty) \to \Rn$, it holds that $f(\gamma(t)) - f(\gamma(0)) = \int_{0}^t {\max_{d \in \D_f(\gamma(s))}}\inner{\dot{\gamma}(s), d} \mathrm{d}s$ for any $t \geq 0$.
It is worth mentioning that the most important choice of the conservative field $\D_f$ is the Clarke subdifferential of $f$. Moreover, as discussed in \cite{bolte2021conservative,castera2021inertial,davis2020stochastic}, the class of path-differentiable functions are general enough to cover a wide range of objective functions in neural network training tasks, especially when the neural networks employ nonsmooth building blocks, such as the ReLU activation function.  
In addition, \cite{bolte2020mathematical,bolte2021conservative} show that the outputs of AD algorithms in differentiating nonsmooth neural networks are enclosed in a conservative field of the loss function. Therefore, the concept of the conservative field is capable of characterizing the outputs of AD algorithms, as they are implemented in training nonsmooth neural networks in practice.

Based on the stochastic approximation frameworks \cite{benaim2006dynamics,benaim2005stochastic,borkar2009stochastic,davis2020stochastic}, several existing works have investigated the convergence properties of stochastic subgradient methods in training  nonsmooth neural networks. In particular, \cite{bolte2021conservative,davis2020stochastic} study the convergence properties of SGD methods and proximal SGD methods for minimizing nonsmooth path-differentiable functions. Moreover, \cite{castera2021inertial} proposes the inertial Newton algorithm (INNA), which can be regarded as a variant of momentum-accelerated SGD method. Additionally, \cite{le2023nonsmooth,ruszczynski2020convergence,xiao2023convergence} establish the convergence properties of SGD methods with heavy-ball momentum. Furthermore, 
\cite{hu2022constraint,hu2022improved} apply these methods to solve manifold optimization problems based on the constraint dissolving approach \cite{xiao2023dissolving}. In addition, \cite{gurbuzbalaban2022stochastic,ruszczynski2021stochastic} design stochastic subgradient methods for solving multi-level nested optimization problems.  

With the concept of conservative field,  the Adam method utilizes the following framework
when applied to solve \eqref{Prob_Ori}:
\begin{equation}
	\label{Eq_intro_Adam}
	\left\{
	\begin{aligned}
		& g_k = d_k + \xi_{k+1},\\
		&\mkp = (1-\theta_k)\mk + \theta_k (g_k+ \sigma \xk),\\
		& \vkp = (1-\beta_k) \vk + \beta_k (g_k + \sigma \xk)^2,\\
		&\xkp = \xk - \eta_k ( \sqrt{\vkp} + \varepsilon)^{-1} \odot \mkp. 
	\end{aligned}
	\right.
\end{equation}
Here $g_k$ is a stochastic subgradient of $f$ at $\xk$, in the sense that $d_k$ represents an inexact evaluation of $\D_f(\xk)$ and $\xi_{k+1}$ is a random vector that characterizes the noise in the evaluation. Moreover, $\odot$ and $(\cdot)^{p}$ refer to the element-wise multiplication and element-wise $p$-th power,  respectively.   The sequences $\{\mk\}$ and $\{\vk\}$ are usually referred to as the momentum terms and estimators respectively,  and they are updated to track the first-order and second-order moment of $g_k+\sigma x_k$. Furthermore, the sequences $\{\eta_k\}$, $\{\theta_k\}$ and $\{\beta_k\}$ are the stepsizes for the variables $\{\xk\}$, the momentum terms $\{\mk\}$ and the estimators $\{\vk\}$, respectively. 

In the framework \eqref{Eq_intro_Adam}, the weight decay term is integrated with the function $f$ throughout the iterations. As a result, we can directly apply the existing convergence results on the Adam method to analyze the convergence properties of the framework \eqref{Eq_intro_Adam}. In particular, when $f$ is a nonsmooth path-differentiable function, \cite{xiao2023adam} investigates the convergence of a class of Adam-family methods based on the frameworks proposed by \cite{benaim2005stochastic,bianchi2022convergence,davis2020stochastic}. However, in the analysis of \cite{xiao2023adam}, the stepsizes are assumed to be diminishing and single-timescale, in the sense that $\{\eta_k\}$, $\{\theta_k\}$ and $\{\beta_k\}$ converge to $0$ at the same rate as $k$ tends to infinity.

In establishing the convergence properties for stochastic subgradient methods, the diminishing stepsizes is a common assumption, as it leads to the almost sure convergence of the iterates $\{\xk\}$ to critical points under various assumptions \cite{benaim2005stochastic,bolte2022subgradient,bolte2021conservative,castera2021inertial,davis2020stochastic,le2023nonsmooth,ruszczynski2020convergence,xiao2023adam,xiao2023convergence}. On the other hand, \cite{bianchi2022convergence} shows that with nonsmooth path-differentiable objective functions and a fixed stepsize,  the iterates of the SGD method only converges to a neighborhood of the $\D_f$-stationary points of $f$ almost surely, even under the noiseless settings. In addition, the results in \cite{bianchi2022convergence} have not been extended to any other stochastic subgradient methods. Given that non-diminishing stepsizes (i.e., ${\lim\inf}_{k\to \infty} \eta_k > 0$)  are widely employed in most computational frameworks, it is thus important for us to investigate the convergence properties of the Adam-family methods in cases where the sequence of stepsizes ${\eta_k}$ is non-diminishing.

\paragraph{\bf Challenges from decoupled weight decay in Adam-family methods}
Another challenge in solving \eqref{Prob_Ori} by Adam-family methods is related to the incorporation of the weight decay term. The conventional approach is to directly minimize $g$ by these Adam-family methods, as implemented in various computational frameworks. In these methods, \cite{loshchilov2017decoupled} demonstrates that the weight decay is coupled with the stochastic subgradients of $f$, in the sense that $f$ and the weight decay term $\frac{\sigma}{2} \norm{x}^2$ are treated as an integrated function to be minimized (e.g., the Adam method in framework \eqref{Eq_intro_Adam}).  

As demonstrated in \cite{loshchilov2017decoupled}, the Adam method with coupled weight decay usually exhibits worse generalization performance than the SGD method. To address this issue, \cite{loshchilov2017decoupled} suggests decoupling the weight decay term from the stochastic subgradients of $f$, and proposes the AdamW method. The update schemes of the AdamW method can be summarized by the following framework:
\begin{equation*}
	\left\{
	\begin{aligned}
		& g_k = d_k + \xi_{k+1},\\
		&\mkp = (1-\theta_k)\mk + \theta_k g_k,\\
		& \vkp = (1-\rho_k) \vk + \rho_k (g_k)^2,\\
		&\xkp = \xk - \eta_k ( \sqrt{\vkp} + \varepsilon)^{-1} \odot \mkp -\eta_k \sigma \xk. 
	\end{aligned}
	\right.
\end{equation*}
Here, \cite{loshchilov2017decoupled} demonstrates that the weight decay is decoupled from the momentum terms $\{\mk\}$ and the estimators $\{\vk\}$, in the sense that the update schemes for $\{\mk\}$ and $\{\vk\}$ are independent of the weight decay parameter $\sigma$. Moreover, unlike the Adam method in \eqref{Eq_intro_Adam}, the weight decay term $\sigma \xk$ is not scaled by the preconditioner $( \sqrt{\vkp} + \varepsilon)^{-1}$ in the AdamW method.

The AdamW method, recognized for its superior generalization performance over the Adam method with coupled weight decay (i.e., the method in \eqref{Eq_intro_Adam}), has become a popular choice in the training of neural networks \cite{loshchilov2017decoupled}, especially in image classification tasks. However, compared with the Adam method, the convergence properties of the AdamW method remain relatively unexplored. As suggested in \cite{loshchilov2017decoupled,zhou2022towards}, the AdamW method iterates by taking a descent step towards a dynamically adjusted surrogate function $f(x) + \frac{\sigma}{2} \inner{x, \sigma (\sqrt{\vkp} + \varepsilon) \odot x}$ in the $k$-th iteration, thereby lacking a clearly defined objective function to minimize. As a result, only the paper \cite{zhou2022towards} has established the convergence properties of the AdamW method for continuously differentiable $f$. In \cite{zhou2022towards}, the stationarity of the AdamW method is measured by $\norm{\nabla f(x) + \sigma (\sqrt{\vkp} + \varepsilon) \odot x}$. As the estimators $\{\vk\}$ evolves over iterations and may not converge, the proposed stationarity measure is at best an approximation of the standard notion of stationarity. More importantly, the analysis in \cite{zhou2022towards} relies on the differentiability of $f$, and cannot be extended to analyze the convergence of AdamW for nonsmooth cases.  Consequently, the results presented in \cite{zhou2022towards} do not sufficiently explain the convergence of AdamW in real-world training tasks, where the neural networks are typically nonsmooth.

Given that Adam-family methods with coupled weight decay usually perform less effectively than the AdamW method, and considering that the AdamW method lacks convergence guarantees in training nonsmooth neural networks, we are driven to raise the following question:
\begin{quote}
	Can we design Adam-family methods with decoupled weight decay that have convergence guarantees with non-diminishing stepsizes $\{\eta_k\}$ under practical settings, especially in the context of training nonsmooth neural networks? 
\end{quote}

\paragraph{\bf Contributions}

The contributions of our paper are summarized as follows.
\begin{itemize}
	\item {\bf A novel framework with decoupled weight decay}
	
	In this paper, motivated by the AdamW method, we propose a novel framework for Adam-family methods with decoupled weight decay (AFMDW), 
	\begin{equation}
		\tag{AFMDW}
		\label{Eq_Framework}
		\left\{
		\begin{aligned}
			& g_k = d_k + \xi_{k+1},\\
			&\mkp = (1-\theta_k)\mk + \theta_k g_k,\\
			&\text{Choose the estimator  $\vkp$},\\
			&\xkp = \xk - \eta_k H(\vkp) \odot (\mkp + \sigma \xk).
		\end{aligned}
		\right.
	\end{equation}
	Here $d_k$ is an approximated evaluation of $\D_f(\xk)$, while $\xi_{k+1}$ refers to the corresponding evaluation noise of $d_{k}$. Therefore, $g_k$ represents the stochastic subgradients of $f$ at $\xk$.  Moreover, the sequences $\{\eta_k\}$ and $\{\theta_k\}$ are stepsizes for the variables $\{\xk\}$ and the momentum terms $\{\mk\}$, respectively. Furthermore, $H: \Rn \to \Rn$ is the mapping that determines how we construct the preconditioner based on $\vkp$ in the framework \eqref{Eq_Framework}. As the framework \eqref{Eq_Framework} is designed to minimize \eqref{Prob_Ori},  both the momentum term $\mkp$ and the weight decay term $\sigma \xk$ are scaled by $H(\vkp)$ in  \eqref{Eq_Framework}, which makes it different from the AdamW method.  
	
	\item {\bf Convergence with non-diminishing $\{\eta_k\}$}

	Under mild assumptions with non-diminishing stepsizes $\{\eta_k\}$, which is more consistent with practical applications, we prove that any cluster point of $\{\xk\}$ lies in the set $\{x \in \Rn: 0\in \D_f(x) + \sigma x\}$, which can be regarded as the stationary points of $g$ in the sense of the conservative field. Moreover, we extend these results to the cases where $\{\eta_k\}$ and $\{\theta_k\}$ are single-timescale, in the sense that $\{\eta_k\}$ and $\{\theta_k\}$ diminish in the same rate. Furthermore, we demonstrate that the framework \eqref{Eq_Framework} encompasses ({see Table \ref{Table_intro_1} for details}) a wide range of Adam-family methods, including SGD, Adam, AMSGrad, AdaBelief, AdaBound, Yogi, hence providing convergence guarantees for these Adam-family methods in training nonsmooth neural networks. 
	

	\item {\bf Asymptotic approximation to SGD method}

	We prove that under mild conditions, almost surely, the sequence $\{\yk\} := \{-\frac{1}{\sigma}\mk\}$ satisfies $\lim_{k\to \infty} \norm{\yk - \xk} = 0$, and follows the inclusion
	\begin{equation}
		\label{Eq_intro_inclusion}
		\ykp - \yk \in -\frac{\theta_k}{\sigma}\left( \D_f(\xk) + \sigma \yk + \xi_{k+1} \right), \qquad \text{for any $k\geq 0$.} 
	\end{equation}
	Consequently, the framework \eqref{Eq_Framework} asymptotically approximates the SGD method, in the sense that the sequence ${\yk}$ can be viewed as a sequence generated by the SGD method with $\norm{\xk - \yk} = 0$. This fact indicates that the weight decay term in our framework \eqref{Eq_Framework} not only introduces the quadratic regularization $\frac{\sigma}{2} \norm{x}^2$ to $f$ in \eqref{Prob_Ori}, but also guides the sequence ${\xk}$ to follow the iterations of SGD methods when $k$ is sufficiently large. Thus in accordance with some existing works \cite{keskar2017improving,zhou2020towards} demonstrating that SGD usually generalizes better than Adam, our analysis lent support to the empirical observation that the decoupled weight decay reduces generalization error for the Adam method.

	\item {\bf Numerical experiments}

	Based on our proposed framework \eqref{Eq_Framework}, we propose a novel method named Adam with Decoupled Weight Decay (AdamD) and establish its convergence guarantees in training nonsmooth neural networks. We conduct numerical experiments in both image classification and language modeling tasks to assess the performance of our proposed AdamD method. The results show that in image classification tasks, AdamD outperforms Adam and performs comparably to AdamW in both generalization and efficiency. In language modeling tasks, it demonstrates similar effectiveness to Adam and outperforms AdamW, highlighting its versatility and effectiveness across different tasks. Additionally, our numerical experiments illustrate that the sequence $\{\norm{\yk - \xk}\}$ tends to $0$, which validates our theoretical analysis that the proposed  AdamD method asymptotically approximates the SGD method. These results further demonstrate the promising potential of our proposed framework \eqref{Eq_Framework}. 
\end{itemize}

\paragraph{\bf Organization}
The rest of this paper is organized as follows. In Section 2, we define the notations used throughout the paper and present the necessary preliminary concepts related to nonsmooth analysis and stochastic approximation. Section 3 presents the convergence properties of our proposed framework \eqref{Eq_Framework} with non-diminishing stepsizes $\{\eta_k\}$. In Section 4, we extend these convergence properties to framework \eqref{Eq_Framework} with single-timescale stepsizes. As an application of our theoretical analysis, we propose a new Adam-family method named Adam with Decoupled Weight Decay (AdamD) and establish its convergence properties in Section 5. In Section 6, we present the results of our numerical experiments that investigate the performance of the proposed AdamD in training nonsmooth neural networks. Some further discussions on the AdamD method are also presented in Section 6. Finally, we conclude the paper in the last section.

\section{Preliminaries}

\subsection{Notations}
For any vectors $x$ and $y$ in $\Rn$ and $\delta \in \bb{R}$, we denote $x\odot y$, $x^{ \delta}$, $x/y$, $|x|$, $x+\delta$, $\sqrt{x}$ as the vectors whose $i$-th entries are given by $x_iy_i$, $x_i^{\delta}$, $x_i/y_i$, $|x_i|$, $x_i + \delta$, and $\sqrt{x_i}$, respectively. 
We 
denote $\Rn_+:=\{ x\in \Rn: x_i\geq 0 \text{ for any } 1\leq i\leq n \}$. Moreover, for any subsets $\ca{X}, \ca{Y} \subset \Rn$, we denote $\ca{X}\odot \ca{Y}:= \{x\odot y: x \in \ca{X}, y\in \ca{Y} \}$,  $|\ca{X}|:= \{|x|: x \in \ca{X}\}$ and $\norm{\ca{X}} = \sup\{ \norm{w} : w\in \ca{X}\}$. In addition, for any $z \in \Rn$, we denote $z + \ca{X} := \{z\} + \ca{X}$ and $z \odot \ca{X} := \{z\} \odot\ca{X}$.

Furthermore, for any positive sequence $\{\theta_k\}$, we define 
$\lambda_0 := 0$, $\lambda_i := \sum_{k = 0}^{i-1} \theta_k$ for $i\geq 1$, and $\Lambda(t) := \sup  \{k \geq 0: t\geq \lambda_k\} $. More explicitly, $\Lambda(t) = p$ if $\lambda_p \leq t < \lambda_{p+1}$ for any $p \geq 0$. In particular, $\Lambda(\lambda_p) = p.$

\subsection{Probability theory}
In this subsection, we present some essential concepts from probability theory, which are necessary for the proofs in this paper. 
\begin{defin}
	Let $(\Omega, \ca{F}, \mathbb{P})$ be a probability space. We say $\{\ca{F}_k\}_{k \in \bb{N}}$ is a filtration  if  $\{\ca{F}_k\}$ is a collection of $\sigma$-algebras that
	satisfies $
	\ca{F}_0 \subseteq \ca{F}_1 \subseteq \cdots \subseteq \ca{F}_{\infty} \subseteq \ca{F}$. 
\end{defin}

\begin{defin}
	We say that a stochastic series $\{\xi_k\}$ is a martingale difference sequence if the following conditions hold,
	\begin{itemize}
		\item The sequence of random vectors $\{\xi_k\}$ is adapted to the filtration $\{ \ca{F}_{k} \}$, 
		\item For each $k \geq 0$, almost surely, it holds that $\bb{E}[|\xi_k|] < \infty$ and $\bb{E}\left[ \xi_k | \ca{F}_{k-1} \right] = 0$. 
	\end{itemize}
	Moreover, we say a martingale difference  sequence $\{\xi_k\}$ is uniformly bounded, if there exists a constant $M_{\xi}> 0$ such that $\sup_{k\geq 0}\norm{\xi_k} \leq  M_{\xi}$. 
\end{defin}

In the following, we present the results in \cite[Proposition 4.4]{benaim2006dynamics}, which controls the weighted summation of any uniformly bounded martingale difference sequence,  and plays a crucial role in establishing the convergence properties for our proposed framework \eqref{Eq_Framework}. 
\begin{prop}[Proposition 4.4 in \cite{benaim2006dynamics}]
	\label{Prop_UB_martingale_difference_sequence}
	Suppose  $\{\theta_k\}$ is a diminishing positive sequence of real numbers that satisfy $\lim_{k \to \infty} \theta_k \log(k) = 0$. 
	Then for any $T > 0$, and any uniformly bounded martingale difference sequence $\{\xi_k\}$, almost surely, it holds that
	\begin{equation}
		\lim_{s \to \infty} \sup_{s\leq i \leq \Lambda(\lambda_s + T)}\norm{ \sum_{k = s}^{i} \theta_k \xi_{k+1}} =0 . 
	\end{equation}
\end{prop}

\subsection{Nonsmooth analysis}
\label{Section_Nonsmooth_Analysis}

In this subsection, we introduce some basic concepts in nonsmooth optimization, especially those related to the concept of the conservative field \cite{bolte2021conservative}. Interested readers could refer to \cite{bolte2021conservative,davis2020stochastic} for more details. 

We begin our introduction on the concept of Clarke subdifferential \cite{clarke1990optimization}, which plays an essential role in characterizing stationarity and development of algorithms for nonsmooth optimization problems.

\begin{defin}[\cite{clarke1990optimization}]
	\label{Defin_Subdifferential}
	For any given locally Lipschitz continuous function $f: \Rn \to \bb{R}$ and any $x \in \Rn$, 
	the Clarke subdifferential $\partial f$ is defined as 
	\begin{equation}
		\partial f(x) := \conv\left( \{ d \in \Rn:  \xk \to x, \nabla f(\xk)\to d  \}   \right).
	\end{equation}
\end{defin}

Next we present a brief introduction on the concept of conservative field, which can be applied to characterize how nonsmooth neural networks are differentiated by the automatic differentiation (AD) algorithms.

\begin{defin}
	A set-valued mapping $\ca{D}: \Rn \rightrightarrows \bb{R}^s$ is a mapping from $\Rn$ to a collection of subsets of $\bb{R}^s$. $\D$ is said to have a closed graph, or is graph-closed if the graph of $\ca{D}$, defined by
	\begin{equation*}
		\mathrm{graph}(\D) := \left\{ (w,z) \in \Rn \times \bb{R}^s: w \in \bb{R}^m, z \in \D(w) \right\},
	\end{equation*}
	is a closed subset of $\Rn \times \bb{R}^s$.  
\end{defin}

\begin{defin}
	A set-valued mapping $\ca{D}: \Rn \rightrightarrows \bb{R}^s$ is said to be locally bounded if, for any $x \in \Rn$, there is a neighborhood $V_x$ of $x$ such that $\cup_{y \in V_x}\ca{D}(y)$ is bounded. 
\end{defin}



Next, we present the definition of conservative field and its corresponding potential function.

\begin{defin}
	An absolutely continuous curve is a continuous mapping $\gamma: \bb{R}_+ \to \Rn $ whose derivative $\gamma'$ exists almost everywhere in $\bb{R}_+$ and $\gamma(t) - \gamma(0)$ equals the Lebesgue integral of $\gamma'$ between $0$ and $t$ for all $t \in \bb{R}_+$, i.e.,
	\begin{equation*}
		\gamma(t) = \gamma(0) + \int_{0}^t \gamma'(u) \mathrm{d} u, \qquad \text{for all $t \in \bb{R}_+$}.
	\end{equation*}
\end{defin}

\begin{defin}[Definition 1 in \cite{bolte2021conservative}]
	\label{Defin_conservative_field}
	Let $\ca{D}$ be a graph-closed set-valued mapping from $\bb{R}^{n}$ to subsets of $\bb{R}^{n}$. We call $\ca{D}$ as a conservative field whenever it has nonempty compact values, and for any absolutely continuous curve $\gamma: [0,1] \to \bb{R}^{n} $ satisfying $\gamma(0) = \gamma(1)$, it holds that 
	\begin{equation}
		\label{Eq_Defin_Conservative_mappping}
		\int_{0}^1 \max_{v \in \ca{D}(\gamma(t)) } \inner{\gamma'(t), v} \mathrm{d}t = 0. 
	\end{equation}
	Here the integral is understood in the Lebesgue sense. 
\end{defin}

It is important to note that any conservative field is locally bounded  \cite[Remark 3]{bolte2021conservative}. We now introduce the definition of potential function corresponding to a conservative field.

\begin{defin}[Definition 2 in \cite{bolte2021conservative}]
	\label{Defin_conservative_field_path_int}
	Let $\ca{D}$ be a conservative field in $\bb{R}^{n}$. Then with any given $x_0 \in \bb{R}^{n}$, we can define a function $f: \Rn \to \bb{R}$ through the path integral
	\begin{equation}
		\label{Eq_Defin_CF}
		\begin{aligned}
			f(x) = &{} f(x_0) + \int_{0}^1 \max_{d \in \ca{D}(\gamma(t)) } \inner{\gamma'(t), d} \mathrm{d}t
			= f(x_0) + \int_{0}^1 \min_{d \in \ca{D}(\gamma(t)) } \inner{\gamma'(t), d} \mathrm{d}t
		\end{aligned}
	\end{equation} 
	for any absolutely continuous curve $\gamma$ that satisfies $\gamma(0) = x_0$ and $\gamma(1) = x$.  The function $f$ is called a potential function for $\ca{D}$, and we also say $\ca{D}$ admits $f$ as its potential function, or that $\ca{D}$ is a conservative field for $f$. 
\end{defin}

The following two lemmas characterize the relationship between conservative field and Clarke subdifferential. 

\begin{lem}[Theorem 1 in \cite{bolte2021conservative}]
	\label{Le_Conservative_as_gradient}
	Let $f:\Rn \to \bb{R}$ be a potential function that admits $\D_f$ as its conservative field. Then $\D_f(x) = \{\nabla f(x)\}$ almost everywhere.  
\end{lem}

\begin{lem}[Corollary 1 in \cite{bolte2021conservative}]
	Let $f:\Rn \to \bb{R}$ be a potential function that admits $\D_f$ as its conservative field. Then $\partial f$ is a conservative field for $f$, and for all $x \in \Rn$, it holds that 
	\begin{equation}
		\partial f(x) \subseteq \conv(\D_f(x)). 
	\end{equation} 
\end{lem}
From the above two lemmas, we can conclude that the concept of conservative field can be regarded as a generalization of Clarke subdifferential. Therefore, conservative field can be applied to characterize  stationarity, as illustrated in the following definition. 
\begin{defin}
	Let $f:\Rn \to \bb{R}$ be a potential function that admits $\D_f$ as its conservative field, then we say $x$ is a $\D_f$-stationary point of $f$ if $0 \in \conv(\D_f(x))$. In particular, we say $x$ is a $\partial f$-stationary point of $f$ if $0 \in \partial f(x)$. 
\end{defin}

As demonstrated in \cite{bolte2021conservative},  a conservative field can be regarded as a generalization of Clarke subdifferential. Therefore, a function is differentiable in the sense of  conservative field if it admits a conservative field for which Definition \ref{Defin_conservative_field_path_int} holds true. Such functions are called path-differentiable \cite[Definition 3]{bolte2021conservative}, and we present the detailed definition as follows. 
\begin{defin}
	Given a locally Lipschitz continuous function $f: \Rn \to \bb{R}$, we say $f$ is path-differentiable if $f$ is the potential function of a conservative field on $\Rn$. 
\end{defin}

It is worth mentioning that the class of path-differentiable functions is general enough to cover the objectives in a wide range of real-world problems. As shown in \cite[Section 5.1]{davis2020stochastic}, any Clarke regular function is path-differentiable. Beyond Clarke regular functions, another important class of path-differentiable functions are functions whose graphs are definable in an $o$-minimal structure \cite[Definition 5.10]{davis2020stochastic}. Usually, the $o$-minimal structure is fixed, and we simply call these functions definable. As demonstrated in \cite{van1996geometric}, any definable function admits a Whitney $C^s$ stratification \cite[Definition 5.6]{davis2020stochastic} for any $s \geq 1$, hence is path-differentiable  \cite{bolte2021conservative,davis2020stochastic}. To characterize the class of definable functions, \cite{davis2020stochastic,bolte2021conservative,bolte2022differentiating} shows that numerous common activation functions and dissimilarity functions are all definable. Furthermore, since definability is preserved under finite summation and composition  \cite{bolte2021conservative,davis2020stochastic},  for any neural network built from definable blocks, its loss function is definable and thus belongs to the class of path-differentiable functions.

Moreover, \cite{bolte2007clarke} shows that any Clarke subdifferential of definable functions is definable. Consequently, for any neural network constructed from definable blocks, the conservative field corresponding to the AD algorithms can be chosen as a definable set-valued mapping formulated by compositing the Clarke subdifferentials of all its building blocks \cite{bolte2021conservative}. The following proposition shows that the definability of $f$ and $\mathcal{D}_f$ leads to the nonsmooth Morse–Sard property \cite{bolte2007clarke} for \eqref{Prob_Ori}. 
\begin{prop}[Theorem 5 in \cite{bolte2021conservative}]
	\label{Prop_definable_regularity}
	Let $f$ be a potential function that admits $\D_f$ as its conservative field. Suppose both $f$ and $\D_f$ are definable over $\Rn$, then the set $\{f(x): 0\in \conv(\D_f(x)) \}$ is finite. 
\end{prop}

\subsection{Differential inclusion and stochastic subgradient methods}
In this subsection, we introduce some fundamental concepts related to the stochastic approximation technique that are essential for the proofs presented in this paper. The concepts discussed in this subsection are mainly from \cite{benaim2005stochastic}. Interested readers could refer to \cite{benaim2006dynamics,benaim2005stochastic,borkar2009stochastic,davis2020stochastic} for more details on the stochastic approximation technique.

\begin{defin}
	For any locally bounded set-valued mapping $\ca{D}: \Rn \rightrightarrows \Rn$ that is nonempty compact convex valued and has closed graph,  we say that an absolutely continuous path $x(t)$ in $\Rn$ is a solution for the differential inclusion 
	\begin{equation}
		\label{Eq_def_DI}
		\frac{\mathrm{d} x}{\mathrm{d}t} \in \ca{D}(x),
	\end{equation}
	with initial point $x_0$ if $x(0) = x_0$, and $\dot{x}(t) \in \ca{D}(x(t))$ holds for almost every $t\geq 0$. 
\end{defin}

\begin{defin}
	\label{Defin_delta_expansion}
	For any given set-valued mapping $\D: \Rn \rightrightarrows \Rn$ and any constant $\delta \geq 0$,  the set-valued mapping $\D^{\delta}$ is defined as
	\begin{equation}
		\D^{\delta}(x) := \{ w \in \Rn: \exists z \in \bb{B}_{\delta}(x), \, \mathrm{dist}(w, \ca{D}(z))\leq \delta \}.
	\end{equation}
\end{defin}

\begin{defin}
	\label{Defin_Lyapunov_function}
	Let $\ca{B} \subset \Rn$ be a closed set. A continuous function $\phi:\Rn \to \bb{R}$ is referred to as a Lyapunov function for the differential inclusion \eqref{Eq_def_DI}, with the stable set $\ca{B}$, if it satisfies the following conditions:
	\begin{enumerate}
		\item For any $\gamma$ that is a solution for \eqref{Eq_def_DI} with $\gamma(0) \in \ca{B}$, it holds that $\phi(\gamma(t)) \leq \phi(\gamma(0))$ for any $t\geq0$.
		\item For any $\gamma$ that is a solution for \eqref{Eq_def_DI} with $\gamma(0) \notin \ca{B}$, it holds that $\phi(\gamma(t)) < \phi(\gamma(0))$ for any $t>0$.
	\end{enumerate}
\end{defin}

The following proposition illustrates that $f$ is a Lyapunov function for the differential inclusion $\frac{\mathrm{d} x}{\mathrm{d}t} \in - \ca{D}_f(x)$. The proof of the following proposition directly follows from \cite{bolte2021conservative}, hence is omitted for simplicity. 
\begin{prop}
	\label{Prop_Lyapunov_SGD}
	Suppose $f$ is a path-differentiable function $f$ that admits $\D_f$ as its conservative field, then $f$ is a Lyapunov function for the differential inclusion $\frac{\mathrm{d} x}{\mathrm{d}t} \in - \ca{D}_f(x)$ with the stable set $\{x \in \Rn: 0 \in \D_f(x)\}$.
\end{prop}


\begin{defin}
	\label{Defin_perturbed_solution}
	We say an absolutely continuous function $\gamma$
	is a perturbed solution to \eqref{Eq_def_DI}  if there exists a locally integrable function $u: \bb{R}_+ \to \Rn$, such that 
	\begin{itemize}
		\item For any $T>0$, it holds that $\lim\limits_{t \to \infty} \sup\limits_{0\leq l\leq T} \norm{\int_{t}^{t+l} u(s) ~\mathrm{d}s} = 0$. 
		\item There exists $\delta: \bb{R}_+ \to \bb{R}$ such that $\lim\limits_{t \to \infty} \delta(t) = 0$ and $\dot{\gamma}(t) - u(t) \in \D^{\delta(t)}(\gamma(t))$. 
	\end{itemize}
\end{defin}

Now consider the sequence $\{\xk\}$ generated by the  following updating scheme,  
\begin{equation}
	\label{Eq_def_Iter}
	\xkp = \xk + \eta_k(d_k + \xi_k),
\end{equation}
where $\{\eta_k\}$ is a diminishing positive sequence of real numbers. 
We define the  (continuous-time) interpolated process of $\{\xk\}$ generated by \eqref{Eq_def_Iter} as follows. 
\begin{defin}
	The  (continuous-time) interpolated process of $\{\xk\}$ generated by \eqref{Eq_def_Iter} is the mapping $w: \bb{R}_+ \to \Rn$ such that 
	\begin{equation}
		w(\lambda_i + s) := x_{i} + \frac{s}{\eta_i} \left( x_{i+1} - x_{i} \right), \quad s\in[0, \eta_i). 
	\end{equation}
	Here $\lambda_0 := 0$, and $\lambda_i := \sum_{k = 0}^{i-1} \eta_k$ for $i\geq 1$.
\end{defin}

The following lemma is an extension of \cite[Proposition 1.3]{benaim2005stochastic}, which allows for inexact evaluations of the set-valued mapping $\D$. It shows that the interpolated process of $\{\xk\}$ from \eqref{Eq_def_Iter} is a perturbed solution of the differential inclusion \eqref{Eq_def_DI}.

\begin{lem}
	\label{Le_interpolated_process}
	Let $\ca{D}: \Rn \rightrightarrows \Rn$ be a locally bounded set-valued mapping that is nonempty compact convex valued with closed graph.
	Suppose the following conditions hold in \eqref{Eq_def_Iter}:
	\begin{enumerate}
		\item For any $T> 0$, it holds that $\lim\limits_{s \to \infty} \sup\limits_{s\leq i \leq \Lambda(\lambda_s + T)}\norm{ \sum_{k = s}^{i} \eta_k \xi_k} =0$. 
		\item There exist a positive sequence $\{\delta_k\}$  such that $\lim_{k\to \infty} \delta_k = 0$ and $d_k \in \D^{\delta_k}(\xk)$.
		\item $\sup_{k \geq 0} \norm{\xk}<\infty$, $\sup_{k \geq 0} \norm{d_k} < \infty$. 
	\end{enumerate}
	Then the interpolated process of $\{\xk\}$ is a perturbed solution for \eqref{Eq_def_DI}. 
\end{lem}

The following theorem summarizes the results in \cite{benaim2005stochastic}, which illustrates the convergence of $\{\xk\}$ generated by \eqref{Eq_def_Iter}. It is worth mentioning that Theorem \ref{The_convergence_beniam} is directly derived from putting \cite[Proposition 3.27]{benaim2005stochastic} and \cite[Theorem 3.6]{benaim2005stochastic} together. Therefore, we omit the proof of Theorem \ref{The_convergence_beniam} for simplicity. 

\begin{theo}
	\label{The_convergence_beniam}
	Let $\ca{D}: \Rn \rightrightarrows \Rn$ be a locally bounded set-valued mapping that is nonempty compact convex valued with closed graph. For any sequence $\{\xk\}$, suppose there exist a continuous function $\phi: \Rn \to \bb{R}$ and a closed subset $\ca{B}$ of $\Rn$ such that
	\begin{enumerate}
		\item $\phi$ is bounded from below, and the set $\{\phi(x) \,:\,x\in\ca{B}\}$
		has empty interior in $\bb{R}$.
		\item $\phi$ is a Lyapunov function for the differential inclusion \eqref{Eq_def_DI} that admits $\ca{B}$ as its stable set. 
		\item The interpolated process of $\{\xk\}$ is a perturbed solution of  \eqref{Eq_def_DI}. 
	\end{enumerate}
	Then any cluster point of $\{\xk\}$ lies in $\ca{B}$, and the sequence $\{\phi(\xk)\}$ converges. 
\end{theo}

Similar results under slightly different conditions can be found in \cite{borkar2009stochastic,davis2020stochastic,duchi2018stochastic}, while some recent works \cite{bianchi2021closed,bolte2022long} focus on analyzing the convergence of \eqref{Eq_def_Iter} under more relaxed conditions. Interested readers could refer to those works for details.

\section{Convergence with Non-diminishing $\{\eta_k\}$}
In this section, we prove the convergence properties of the framework \eqref{Eq_Framework} even though the sequence of stepsizes $\{\eta_k\}$ is assumed to be non-diminishing. 

\subsection{Convergence to $\D_g$-stationary points}

We first make the following assumptions on the quadratically regularized optimization problem \eqref{Prob_Ori}. 
\begin{assumpt}
	\label{Assumption_f}
	\begin{enumerate}
		\item $f$ is a path-differentiable function that admits a convex-valued set-valued mapping $\D_f$ as its conservative field. 
		\item The set $\{g(x): 0 \in \D_f(x) + \sigma x\}$ has empty interior in $\bb{R}$. 
		\item The function $g$ is bounded from below over $\Rn$. That is, $\inf_{x \in \Rn} g(x) >-\infty$. 
	\end{enumerate}
\end{assumpt}
As discussed in Section 2.3, the class of path-differentiable functions covers a great number of objective functions in real-world applications. In particular, for a wide range of common neural networks, their loss functions are definable and thus path-differentiable,  as demonstrated in \cite{bolte2021conservative,castera2021inertial,davis2020stochastic}.  As a result, Assumption \ref{Assumption_f}(1) is mild in practice. Moreover, Assumption \ref{Assumption_f}(2) is referred to as the nonsmooth weak Sard's property, which is commonly observed in various existing works \cite{bianchi2021closed,bolte2022subgradient,bolte2021conservative,castera2021inertial,davis2020stochastic,le2023nonsmooth} and is shown to be mild as demonstrated in \cite{bolte2021conservative,castera2021inertial,davis2020stochastic}.

Notice that the chain rule holds for conservative fields \cite[Lemma 5]{bolte2021conservative}, and it is easy to verify that $g$ is a path-differentiable function that admits $\D_f(x) + \sigma x$ as its conservative field. Therefore, in the rest of the paper, we fix the conservative field $\D_g:\Rn \rightrightarrows \Rn$ for the objective function $g$ in \eqref{Prob_Ori}  as: 
\begin{equation}
	\D_g(x) := \D_f(x) + \sigma x. 
\end{equation}

In the following lemma, we present some basic properties of $\D_g$. The proof of Lemma \ref{Le_Dg} straightforwardly follows from \cite[Corollary 4]{bolte2021conservative}, hence it is omitted for simplicity.  
\begin{lem}
	\label{Le_Dg}
	Suppose Assumption \ref{Assumption_f} holds. Then $g$ is a path-differentiable function, and $\D_g$ is a convex-valued graph-closed conservative field that admits $g$ as its potential function. 
\end{lem}

We also need the following assumptions on the framework \eqref{Eq_Framework} for establishing its convergence properties.
\begin{assumpt}
	\label{Assumption_framework}
	\begin{enumerate}
		\item There exists constants $0< \varepsilon_v< M_v$ such that $ \varepsilon_v\leq H(\vk) \leq M_v$ holds for any $k \geq 0$. 
		\item The sequence $\{\xk\}$ is uniformly bounded almost surely. That is, there exists a constant $M_x$ such that  $\max\left\{\norm{m_0}, \sup_{k\geq 0} \norm{\xk} \right\}\leq M_x$ holds almost surely. 
		\item The sequences of stepsizes $\{\eta_k\}$ and $\{\theta_k\}$ are positive and satisfy
		\begin{equation}
			\inf_{k\geq 0} \eta_k > 0, \quad \sup_{k\geq 0} \eta_k < \frac{2}{\sigma M_v}, \quad \sum_{k = 0}^{\infty} \theta_k = \infty, \quad \lim_{k\to \infty} \theta_k \log(k) = 0. 
		\end{equation} 
		\item There exists a non-negative sequence $\{\delta_k\}$ such that $\lim_{k\to \infty} \delta_k = 0$ and $d_k \in \D_f^{\delta_k}(\xk)$. 
		\item  The sequence of noises $\{\xi_k\}$ is a uniformly bounded martingale difference sequence. That is, there exists a constant $M_{\xi}$ such that almost surely, $\sup_{k\geq 0} \norm{\xi_k} \leq M_{\xi}$, and $\bb{E}[\xi_{k+1}|\ca{F}_k] = 0$  for any $k\geq 0$. 
	\end{enumerate}
\end{assumpt}

Here we make some comments to the assumptions in Assumption \ref{Assumption_framework}. Assumption \ref{Assumption_framework}(1)-(2) assumes the uniform boundedness of $\{H(\vk)\}$ and $\{\xk\}$, which is a common assumption in various existing works \cite{benaim2005stochastic,bolte2021conservative,castera2021inertial}. In addition, later in Section 3.2, we provide some sufficient conditions that guarantee the validity of Assumption \ref{Assumption_framework}(1)-(2).  Assumption \ref{Assumption_framework}(3) requires the stepsizes $\{\eta_k\}$ to be non-diminishing, while assumes that $\{\theta_k\}$ is diminishing in the rate of $o(1/\log(k))$. Since $1/\log(k)$ decays very slowly throughout the iterations, the assumptions on the stepsizes $\{\eta_k\}$ and $\{\theta_k\}$  are mild in practice. Assumption \ref{Assumption_framework}(4) characterizes how $d_k$ approximates $\D_f(\xk)$. Furthermore, Assumption \ref{Assumption_framework}(5) assumes that the evaluation noises $\{\xi_k\}$ is a  uniformly bounded martingale difference sequence. As demonstrated in \cite{bolte2021conservative,castera2021inertial}, Assumption \ref{Assumption_framework}(5)  holds when $f$ follows a finite-sum formulation, hence it is mild in practical applications of \eqref{Prob_Ori}.

We begin our theoretical analysis with Lemma \ref{Le_UB_mk_gk}, which shows that the sequence $\{\mk\}$ and $\{g_k\}$ are uniformly bounded.  Lemma \ref{Le_UB_mk_gk} directly follows from the uniform boundedness of $\{\xk\}$ and $\{\xi_k\}$ and the fact that $\D_f$ is locally bounded, hence we omit its proof for simplicity. 
\begin{lem}
	\label{Le_UB_mk_gk}
	Suppose Assumption \ref{Assumption_f} and Assumption \ref{Assumption_framework} hold. Then there exists a constant $M_d>0$ such that  $\sup_{k\geq 0} \{\norm{g_k} + \norm{\mk}\} \leq M_d$ holds almost surely. 
\end{lem}

The Lemma \ref{Le_xk_yk_close}  illustrates that $\norm{\sigma \xk + \mk} \to 0$ as $k$ tends to infinity. 
\begin{lem}
	\label{Le_xk_yk_close}
	Suppose Assumption \ref{Assumption_f} and Assumption \ref{Assumption_framework} hold. Then there exists a nonnegative sequence $\{\hat{\delta}_k\}$ such that  $\lim_{k\to \infty} \hat{\delta}_k = 0$ and almost surely, $\norm{\sigma \xk + \mk} \leq \hat{\delta}_k$ holds for any $k\geq 0$. 
\end{lem}
\begin{proof}
	Let $\eta_{\min}:=\inf_{k\geq 0} \eta_k$ and $\eta_{\max} := \sup_{k\geq 0} \eta_k$. Then from Assumption \ref{Assumption_framework}(3), there exists a constant $\tilde{\eta} > 0$ such that $ \max\{|1-\eta_k \sigma M_v|, |1-\eta_k \sigma \varepsilon_v|\} \leq 1-\tilde{\eta}$ holds for any $k\geq 0$. 
	Then from the update scheme of $\{\xk\}$ in framework \eqref{Eq_Framework}, it holds that 
	\begin{equation}
		\begin{aligned}
			&\norm{\sigma \xkp + \mkp} \\
			={}& \norm{(1-\eta_k \sigma H(\vkp) )\odot(\sigma \xk + \mk)\, +\,
				\theta_k (1-\eta_k \sigma H(\vkp)) \odot (g_k -\mk ) }\\
			\leq{}& \max\{|1-\eta_k \sigma M_v|, |1-\eta_k \sigma \varepsilon_v|\}   (\norm{\sigma \xk + \mk} + \theta_k \norm{g_k - m_k})\\
			\leq{}& (1-\tilde{\eta}) \norm{ \sigma \xk + \mk} + 2M_d\theta_k  \leq  (1-\tilde{\eta})^{k+1}
			\norm{\sigma x_0 + m_0}+ 2M_d \sum_{i = 0}^{k} (1-\tilde{\eta})^{k-i}\theta_i \\
			\leq{}& (1-\tilde{\eta})^{k+1}(\sigma M_x + M_d)+ 
			2M_d\sum_{i = 0}^{k} (1-\tilde{\eta})^{k-i}\theta_i 
			\; =: \; \hat{\delta}_k.
		\end{aligned}
		\label{xk_yk_est}
	\end{equation}
	Since $\lim_{k\to \infty} \theta_k = 0$, it holds that $\lim\limits_{k\to \infty} \sum_{i = 0}^{k} (1-\tilde{\eta})^{k-i}\theta_i = 0$. 
	Thus we get $\lim_{k\to \infty} \hat{\delta}_k = 0$, and $\norm{\sigma \xk + \mk} \leq \hat{\delta}_k$ holds for any $k\geq 0$. This completes the proof. 
\end{proof}

Based on the Lemma \ref{Le_xk_yk_close}, let the auxiliary sequence $\{\yk\}$ be defined as 
\begin{equation}
	\label{Eq_Defin_yk}
	\yk := -\frac{1}{\sigma} \mk, \quad \text{ for any $k\geq 0$.}
\end{equation}
Then we can conclude that $\lim_{k\to \infty} \norm{\xk - \yk} = 0$ almost surely. More importantly, substituting \eqref{Eq_Defin_yk} into the update scheme for $\{\mk\}$ in \eqref{Eq_Framework}, we arrive at the following inclusion
\begin{equation}
	\label{Eq_update_yk}
	\begin{aligned}
		\ykp ={}& \yk - \frac{\theta_k}{\sigma} \left( d_k + \sigma \yk + \xi_{k+1}  \right). 
	\end{aligned}
\end{equation}

In the following lemma, we prove that $d_k + \sigma \yk$ can be regarded as an approximated evaluation for $\D_g(\yk)$. 

\begin{lem}
	\label{Le_approx_dk}
	Suppose Assumption \ref{Assumption_f} and Assumption \ref{Assumption_framework} hold. Then let ${\delta}_k^\star :=  (1+ \sigma)(\delta_k + \hat{\delta}_k)$, it holds that 
	\begin{equation}
		d_k + \sigma y_k \in \D_g^{{\delta}_k^\star}(\yk). 
		\label{Eq_tmp}
	\end{equation}
\end{lem}
\begin{proof}
	As illustrated in Assumption \ref{Assumption_framework}(4), there exists $\tilde{x}_k \in \bb{B}_{\delta_k}(\xk)$ and $\tilde{d}_k \in \D_f(\tilde{x}_k)$ such that $\norm{d_k -\tilde{d}_k} \leq \delta_k$. Combining with Lemma \ref{Le_xk_yk_close}, it holds that $\norm{\yk - \tilde{x}_k} \leq \norm{\yk - \xk} + \norm{\xk - \tilde{x}_k} \leq  \hat{\delta}_k + \delta_k $. As a result, 
	\begin{eqnarray*}
		\mathrm{dist}\left(d_k + \sigma y_k, \D_g(\tilde{x}_k)  \right) 
		&\leq& \norm{d_k+\sigma y_k - (\tilde{d}_k + \sigma \tilde{x}_k)}
		\\
		&\leq& \norm{d_k - \tilde{d}_k} + \sigma \norm{y_k -\tilde{x}_k}
		\leq \delta_k+\sigma(\delta_k +\hat{\delta}_k). 
	\end{eqnarray*}
	Since $\tilde{x}_k \in \bb{B}_{\delta_k^\star}(y_k)$ and $\mathrm{dist}(d_k + \sigma y_k, \D_g(\tilde{x}_k)) \leq \delta_k^\star$, we get \eqref{Eq_tmp}.
\end{proof}

We can conclude from Lemma \ref{Le_approx_dk} that the auxiliary sequence $\{\yk\}$ follows the differential inclusion,
\begin{equation}
	\ykp \in \yk - \frac{\theta_k}{\sigma} \left( \D_g^{\delta^{\star}_k}(\yk) + \xi_{k+1}  \right). 
\end{equation}
This fact illustrates that the sequence $\{\yk\}$ can be viewed as a sequence generated by the SGD method for minimizing $g$. Therefore, in the following proposition, we prove that the interpolated process of the sequence $\{\yk\}$ is a perturbed solution of the following differential inclusion:
\begin{equation}
	\label{Eq_DI_SGD_yk}
	\frac{\mathrm{d}y}{\mathrm{d}t} \in - \D_g(y). 
\end{equation}

\begin{prop}
	\label{Prop_perturbed_solution_yk}
	Suppose Assumption \ref{Assumption_f} and Assumption \ref{Assumption_framework} hold. 
	Then the interpolated process of the sequence $\{\yk\}$ is a perturbed solution for the differential inclusion \eqref{Eq_DI_SGD_yk}. 
\end{prop}
\begin{proof}
	Based on Lemma \ref{Le_interpolated_process}, by verifying its conditions, we can prove that the interpolated process of $\{\yk\}$ is a perturbed solution for the differential inclusion \eqref{Eq_DI_SGD_yk}. 
	
	Condition (1) of Lemma \ref{Le_interpolated_process} directly follows from Assumption \ref{Assumption_framework}(5) and Proposition \ref{Prop_UB_martingale_difference_sequence}, by choosing the stepsizes in \eqref{Eq_def_Iter} as $\{\frac{\theta_k}{\sigma}\}$. Moreover, Lemma \ref{Le_approx_dk} guarantees the validity of the condition (2) in Lemma \ref{Le_interpolated_process}. Furthermore, condition (3) of Lemma \ref{Le_interpolated_process} follows from Assumption \ref{Assumption_framework}(2) and Lemma \ref{Le_UB_mk_gk}. As a result, directly from Lemma \ref{Le_interpolated_process}, we can conclude that almost surely, the interpolated process of $\{\yk\}$ is a perturbed trajectory of the differential inclusion \eqref{Eq_DI_SGD_yk}. 
\end{proof}

In the following theorem, we prove the convergence properties of the framework \eqref{Eq_Framework}.
\begin{theo}
	\label{The_convergence_Nondiminishing}
	Suppose Assumption \ref{Assumption_f} and Assumption \ref{Assumption_framework} hold. 
	Then almost surely, any cluster point of the sequence $\{\xk\}$ is a $\D_g$-stationary point of $g$, and $\{g(\xk)\}$ converges.  
\end{theo}
\begin{proof}
	From Lemma \ref{Le_Dg} and Proposition \ref{Prop_Lyapunov_SGD}, we can conclude that $g$ is a Lyapunov function for the differential inclusion 
	\eqref{Eq_DI_SGD_yk}
	with the stable set $\{x \in \Rn: 0\in \D_g(x)\}$. Moreover, Proposition \eqref{Prop_perturbed_solution_yk} illustrates that almost surely, the interpolated process of the sequence $\{\yk\}$ is a perturbed solution of the differential inclusion \eqref{Eq_DI_SGD_yk}. As a result, it follows from Theorem \ref{The_convergence_beniam} that any cluster point of $\{\yk\}$ lies in the set $\{x \in \Rn: 0\in \D_g(x)\}$ and the sequence $\{g(\yk)\}$ converges. 
	
	Finally, Lemma \ref{Le_xk_yk_close} illustrates that $\lim_{k\to \infty}\norm{\xk - \yk} = 0$ holds almost surely. Then from the continuity of $g$ and the convergence properties of $\{\yk\}$,  we can conclude that any cluster point of $\{\xk\}$ lies in the set $\{x \in \Rn: 0\in \D_g(x)\}$ and the sequence $\{g(\xk)\}$ converges. This completes the proof. 
\end{proof}

\subsection{Comments on the uniform boundedness of $\{\xk\}$ and $\{\vk\}$}

In this subsection, we present some sufficient and easy-to-verify conditions that guarantee the validity of Assumption \ref{Assumption_framework}(2) and Assumption \ref{Assumption_framework}(3). The following proposition illustrates that with some mild global continuity condition for $f$ and the uniform boundedness of $\{H(\vk)\}$, the sequence $\{\xk\}$ is uniformly bounded. and thus satisfies Assumption \ref{Assumption_framework}(2).

\begin{prop}
	\label{Prop_UB}
	Suppose Assumption \ref{Assumption_f} and Assumption \ref{Assumption_framework}(1)(3)(4)(5)  hold. Moreover, suppose there exists constants $L>0$ and $\nu \in [0,1)$ 
	such that $\norm{ \D_f(x)} \leq L(1+\norm{x}^{\nu})$ holds. Then for any initial point $(x_0, m_0, v_0)$, there exists a constant $Q>0$ such that $\sup_{k\geq 0} \norm{\xk} \leq Q$. 
\end{prop}
\begin{proof}
	As illustrated in Assumption \ref{Assumption_framework}, $d_k \in \D_f^{\delta_k}(\xk)$ and $\{\xi_k\}$ is uniformly bounded. Then it is easy to verify that there exists a constant $\hat{L}$ such that $\norm{g_k} = \norm{d_k + \xi_{k+1}} \leq \hat{L}(1+ \norm{\xk}^{\nu})$ holds for any $k\geq 0$. 
	
	Let the constant $Q$ be defined as 
	\begin{equation}
		Q \geq \max\left\{ \left(\frac{2 M_v  \hat{L} }{\varepsilon_v \sigma }\right)^{\frac{1}{1-\nu}},  \frac{M_v \norm{m_0}}{ \varepsilon_v \sigma}, \norm{x_0}+1 \right\}. 
	\end{equation}
	
	In the following, for any sequence $\{\xk\}$ generated from \eqref{Eq_Framework}, we aim to prove that the set $\{k\geq 0: \norm{\xk} \geq Q  \}$ is an empty set by contradiction. Therefore, we assume that the set $\{k\geq 0: \norm{\xk} \geq Q \}$ is non-empty and set $\tau = \inf \{k\geq 0: \norm{\xk} \geq Q \}-1$. Then from the definition of $\tau$, we have $\norm{x_{\tau + 1}} > \norm{x_{\tau}}$. 
	
	On the other hand, from the update scheme \eqref{Eq_Framework}, for any $k\leq \tau$, we have
	\begin{equation*}
		\norm{\mkp} \leq \max\{m_0, \sup_{0\leq i\leq k+1} \norm{g_k}\} \leq \max\{ \norm{m_0},  \hat{L}(1+Q^{\nu}) \} \leq \frac{\sigma \varepsilon_v}{M_v} Q ,
	\end{equation*}
	where the last inequality follows from the definition of $Q$ and the fact that
	\[
	\hat L(1+Q^\nu)\leq 2\hat LQ^\nu=\frac{\sigma\varepsilon_v}{M_v}\cdot\frac{2M_v\hat L}{\sigma\varepsilon_v}Q^\nu\leq\frac{\sigma\varepsilon_v}{M_v}Q^{1-\nu}Q^\nu=\frac{\sigma\varepsilon_v}{M_v}Q.
	\]
	Then it holds that 
	\begin{equation*}
		\begin{aligned}
			&\norm{x_{\tau + 1}} = \norm{(1-\eta_k \sigma H_{\tau}(v_{\tau+1})) \odot x_{\tau} - \eta_k H_{\tau}(v_{\tau+1}) \odot m_{\tau + 1} }\\
			\leq{}& (1-\eta_{k} \sigma \varepsilon_v)\norm{x_{\tau}} + \eta_{k} M_v \norm{m_{\tau +1}} \leq  (1-\eta_{k}\sigma \varepsilon_v) Q + \eta_{k} M_v \cdot \frac{\sigma \varepsilon_v}{M_v} Q \leq Q. 
		\end{aligned}
	\end{equation*}
	But $\norm{x_{\tau +1}} \leq Q$ contradicts to the definition of $\tau$. Therefore, we can conclude that the set $\{k\geq 0: \norm{\xk} \geq Q \}$ is empty.  Therefore, we derive that $\sup_{k\geq 0} \norm{\xk} \leq Q$ 
	holds almost surely. This completes the proof.    
\end{proof}

\begin{rmk}
	It is worth mentioning that the proof of Proposition \ref{Prop_UB} does not require the positiveness of $\eta_{\min}$ in Assumption \ref{Assumption_framework}. Therefore, when 
	the stepsizes $\{\eta_k\}$ in the framework \eqref{Eq_Framework} is diminishing, we can still prove the uniform boundedness of $\{\xk\}$ by the same proof techniques as those in Proposition \ref{Prop_UB}. 
\end{rmk}

Then we discuss the uniform boundedness of the sequence $\{H(\vk)\}$. Apart from Assumption \ref{Assumption_f} and Assumption \ref{Assumption_framework}, we make the assumption on the global Lipschitz continuity of $f$, in the sense that 
\begin{equation}
	\label{Eq_assumption_Lip}
	\sup_{x \in \Rn} \norm{\D_f(x)} \leq M_f, \qquad \text{for some constant $M_{f}>0$. } 
\end{equation}
Such an assumption is standard in various existing works. Table \ref{Table_intro_1} lists some Adam-family methods, where the sequence $\{\vk\}$ remains uniformly bounded under Assumption \ref{Assumption_f}, Assumption \ref{Assumption_framework}(3)-(5), and \eqref{Eq_assumption_Lip}.

\begin{table}[htbp]
	\centering
	\footnotesize
	\caption{Different update schemes for $\{\vk\}$ in the framework \eqref{Eq_Framework} under Assumption \ref{Assumption_f}, Assumption \ref{Assumption_framework}(3)-(5), and \eqref{Eq_assumption_Lip}. Here $\varepsilon, c_{l}, c_{u} > 0$ are hyper-parameters for these Adam-family methods. }
	\label{Table_intro_1}
	\begin{tabular}{c|ccc}
		\hline
		Method & Update scheme for $\{\vk\}$ & Formulation for $H(v)$ &  Choice of $(\varepsilon_v, M_v)$\\ \hline
		SGDW \cite{loshchilov2017decoupled} & $\vkp = (1-\beta_1) \vk +\beta_1g_k^2$ & $1$ & $(1,1)$ \\
		Adam \cite{kingma2014adam} & $\vkp = (1-\beta_1) \vk + \beta_1 g_k^2$ & $(\sqrt{v} + \varepsilon)^{-1}$ & $(\frac{1}{ (M_f + M_{\xi}) +\varepsilon}, \frac{1}{\varepsilon})$ \\
		AMSGrad \cite{reddi2019convergence}& $\vkp = \max\{ \vk,  (1-\beta_1) \vk + \beta_1 g_k^2\}$ & $(\sqrt{v} + \varepsilon)^{-1}$ & $(\frac{1}{ (M_f + M_{\xi}) +\varepsilon}, \frac{1}{\varepsilon})$ \\
		Adamax \cite{kingma2014adam} & $\vkp = \max\{\beta_1 \vk, |g_k| + \varepsilon\}$ & $ (v)^{-1}$ &  $(\frac{1}{ (M_f + M_{\xi})^2 +\varepsilon}, \frac{1}{\varepsilon})$ \\
		RAdam \cite{liu2019variance}& $\vkp = (1-\beta_1) \vk + \beta_1 g_k^2$ & $(\sqrt{v} + \varepsilon)^{-1}$ & $(\frac{1}{ (M_f + M_{\xi}) +\varepsilon}, \frac{1}{\varepsilon})$ \\
		AdaBelief \cite{zhuang2020adabelief}& $\vkp = (1-\beta_1) \vk + \beta_1 (g_k-\mkp)^2$ & $(\sqrt{v} + \varepsilon)^{-1}$ & $(\frac{1}{ 2(M_f + M_{\xi}) +\varepsilon}, \frac{1}{\varepsilon})$ \\
		AdaBound \cite{luo2019adaptive} & $\vkp = (1-\beta_1) \vk + \beta_1 g_k^2$  & $\min\{c_{l}, \max\{c_{u}, v^{-\frac{1}{2}} \} \}$ & $(c_{l}, c_{u})$ \\
		Yogi \cite{zaheer2018adaptive} & $\vkp = \vk - \beta_1 \mathrm{sign}(\vk-g_k^2) \odot g_k^2$ &  $(\sqrt{v} + \varepsilon)^{-1}$ & $(\frac{1}{ (M_f + M_{\xi}) +\varepsilon}, \frac{1}{\varepsilon})$ \\
		\hline
	\end{tabular}
\end{table}

Then based on the discussions above, we have the following corollary illustrating the convergence properties of $\{\xk\}$ under easy-to-verify conditions. 
\begin{coro}
	Suppose Assumption \ref{Assumption_f} and Assumption \ref{Assumption_framework}(3)-(5) hold. Moreover, we assume that the conservative field $\D_f$ satisfies  $\sup_{x \in \Rn} \norm{\D_f(x)} \leq M_f$ for some constant $M_f > 0$.  
	Then almost surely, any cluster point of the sequence $\{\xk\}$ is a $\D_g$-stationary point of $g$, and $\{g(\xk)\}$ converges.  
\end{coro}

\section{Convergence with Single-timescale Stepsizes}

In this section, we investigate the convergence of the framework \eqref{Eq_Framework} when the sequences of stepsizes  $\{\eta_k\}$ and $\{\theta_k\}$ are single-timescale in the sense that they diminish at the same rate. 

The convergence properties presented in Section 3 suggest that the sequence ${\yk}$ asymptotically approximates the trajectories of the differential inclusion \eqref{Eq_DI_SGD_yk}. One may conjecture that this phenomenon is attributable to the involvement of non-diminishing stepsizes $\{\eta_k\}$ in the framework \eqref{Eq_Framework}. 

However, in this section, we aim to show that when single-timescale stepsizes are employed in the framework \eqref{Eq_Framework},  the interpolated process of $\{\yk\}$  is still a perturbed sequence of the differential inclusion \eqref{Eq_DI_SGD_yk}.   These theoretical results suggest that it is the decoupled weight decay that leads to the asymptotic approximation of the differential inclusion \eqref{Eq_DI_SGD_yk} in the framework \eqref{Eq_Framework}, regardless of the timescale of the employed stepsizes $\{\eta_k\}$ and $\{\theta_k\}$.   

The proof techniques in this section are motivated by the techniques in \cite[Section 3]{xiao2023adam}. To prove the convergence of \eqref{Eq_Framework} with single-timescale stepsizes, we need to make the following assumptions. 
\begin{assumpt}
	\label{Assumption_framework_ST}
	\begin{enumerate}
		\item The sequence $\{\xk\}$ is uniformly bounded almost surely. That is, there exists a constant $M_x$ such that  $\max\left\{\norm{m_0}, \sup_{k\geq 0} \norm{\xk} \right\}\leq M_x$ holds almost surely. 
		\item There exists a locally bounded  mapping $W: \Rn\times \Rn \times \Rn_+ \to \Rn_+$ and a prefixed constant $\tau_2 >0$ such that the sequence of estimators $\{\vk\}$ follows the update scheme $\vkp = \vk - \tau_2\eta_k (\vk - W(g_k, \mkp))$. 
		\item The mapping $H:\Rn_+ \to \Rn_+$ is fixed as $H(v) = (\max\{v, 0\}+\varepsilon)^{-\frac{1}{2}}$ for a prefixed constant $\varepsilon > 0$. 
		\item The sequences of stepsizes $\{\eta_k\}$ and $\{\theta_k\}$ are positive and satisfies
		\begin{equation}
			\sum_{k = 0}^{\infty} \theta_k = \infty, \quad \lim_{k\to \infty} \theta_k \log(k) = 0, \quad \lim_{k\to\infty} \frac{\theta_k}{\eta_k} = \tau_1,
		\end{equation}
		for a prefixed positive constant $\tau_1 \in [\frac{\tau_2}{4},\infty)$. 
		\item There exists a non-negative sequence $\{\delta_k\}$ such that $\lim_{k\to \infty} \delta_k = 0$ and $d_k \in \D_f^{\delta_k}(\xk)$. 
		\item   The sequence of noises $\{ \xi_k \}$ is a uniformly bounded martingale difference sequence. 
	\end{enumerate}
\end{assumpt}

Here are some comments for Assumption \ref{Assumption_framework_ST}. Assumption \ref{Assumption_framework_ST}(1)(5)(6) are identical to Assumption \ref{Assumption_framework}(1)(4)(5), respectively. Assumption \ref{Assumption_framework_ST}(2) characterizes how the estimators $\{\vk\}$ are updated. As discussed in \cite{barakat2021convergence,xiao2023adam}, Assumption \ref{Assumption_framework_ST}(2) is general enough to enclose the update schemes for Adam,  AdaBelief, AMSGrad, and Yogi. Moreover, Assumption \ref{Assumption_framework_ST}(3) fixes the formulation of the mapping $H$, and Assumption \ref{Assumption_framework_ST}(4) assumes that the stepsizes in framework \eqref{Eq_Framework} are of single-timescale.


We begin our analysis with the following lemma, which shows the uniform boundedness of $\{\mk\}$ and $\{g_k\}$ directly from the uniform boundedness of $\{\xk\}$ in Assumption \ref{Assumption_framework_ST}(1). As a result, we omit its proof for simplicity. 
\begin{lem}
	\label{Le_UB_mk_gk_ST}
	Suppose Assumption \ref{Assumption_f} and Assumption \ref{Assumption_framework_ST} hold. Then there exists a constant $M_d>0$ such that  $\sup_{k\geq 0} \norm{g_k} + \norm{\mk} \leq M_d$  holds almost surely. 
\end{lem}

We then present the following auxiliary lemma, which directly follows from the uniform boundedness of $\{\xk\}$, $\{\mk\}$ and $\{g_k\}$ in Lemma \ref{Le_UB_mk_gk_ST}, together with the local boundedness of the mappings $\D_f$ and $W$. 
\begin{lem}
	\label{Le_UB_W_ST}
	Suppose Assumption \ref{Assumption_f} and Assumption \ref{Assumption_framework_ST} hold. Then there exists a constant $M_{W}>0$ such that $\sup_{k\geq 0} \norm{W(g_k, \mkp)} \leq  M_{W}$ holds almost surely. 
\end{lem}

Let $\ca{P}_+(v) := \max\{v, 0\}$, and $\ca{U}(x, m) := \{d \in \Rn_+: \norm{d} \leq M_{W}\}$. Consider the set-valued mapping $\ca{G}: \Rn \times \Rn \times \Rn \rightrightarrows \Rn \times \Rn \times \Rn$ defined by 
\begin{equation}
	\label{Eq_mapping_G}
	\ca{G}(x,m,v) := \left[\begin{matrix}
		(\ca{P}_+(v)  + \varepsilon )^{-\frac{1}{2}} \odot \left(   m + \sigma x \right)\\
		\tau_1 m - \tau_1 \D_f(x)\\
		\tau_2 v - \tau_2  \ca{U}(x,m)\\
	\end{matrix}\right], 
\end{equation}
and the following differential inclusion:
\begin{equation} 
	\label{Eq_DI_ST}
	\left(\frac{\mathrm{d}x}{\mathrm{d}t}, \frac{\mathrm{d}m}{\mathrm{d}t}, \frac{\mathrm{d}v}{\mathrm{d}t}\right)
	\in 
	-\ca{G}(x,m,v) .
\end{equation}

In the following lemma, we prove that the set-valued mapping $\ca{G}$ is capable of characterizing the update direction of $\{(\xk, \mk, \vk)\}$ in the framework \eqref{Eq_Framework}. The proof straightforwardly follows from Lemma \ref{Le_UB_W_ST}, hence we omit it for simplicity. 
\begin{lem}
	\label{Le_inclusion_vk_ST}
	Suppose Assumption \ref{Assumption_f} and Assumption \ref{Assumption_framework_ST} hold. Then the inclusion 
	\begin{equation}
		\vkp \in \vk - \tau_2\eta_k (\vk - \ca{U}(\xk, \mk))
	\end{equation}
	holds for any $k\geq 0$. Furthermore, $\sup_{k\geq 0} \norm{\vkp} <\infty$ holds almost surely. 
\end{lem}

Let $\partial \ca{P}_+$ be the generalized Jacobian of the mapping $\ca{P}_+$, and define  the function $h: \Rn \times \Rn \times \Rn \to \bb{R}$ as 
\begin{equation}
	\label{Eq_Lyapunov_func_ST}
	h(x,m,v) = f(x) + \frac{\sigma}{2} \norm{x}^2 + \frac{1}{2\tau_1} \inner{m +\sigma x, (\ca{P}_+(v)  + \varepsilon )^{-\frac{1}{2}} \odot (m+\sigma x)}. 
\end{equation}
The next Lemma \ref{Le_h_conservative_field_ST} presents the formulation of the conservative field of $h$.

\begin{lem}
	\label{Le_h_conservative_field_ST}
	Suppose Assumption \ref{Assumption_f}  and Assumption \ref{Assumption_framework_ST} hold. Then $h$ is a potential function that admits 
	\begin{equation}
		\D_h(x, m, v) = 
		\left[\begin{matrix}
			\D_f(x) + \sigma x + \frac{\sigma}{\tau_1}(\ca{P}_+(v) +\varepsilon)^{-\frac{1}{2}} {\odot} (m+\sigma x)\\
			\frac{1}{\tau_1} (\ca{P}_+(v) + \varepsilon )^{-\frac{1}{2}} \odot (m +\sigma x) \\
			-\frac{1}{4\tau_1} (m +\sigma x)^2 \odot (\ca{P}_+(v) +\varepsilon)^{-\frac{3}{2}} \odot \partial \ca{P}_+(v)
		\end{matrix}\right]
	\end{equation}
	as its conservative field. 
\end{lem}
\begin{proof}
	Notice that $f$ is a potential function that admits $\D_f$  as its conservative field, and the function $(x,m,v)\mapsto \inner{m +\sigma x, (\ca{P}_+(v)  + \varepsilon )^{-\frac{1}{2}} \odot (m+\sigma x)}$ is semi-algebraic and thus definable. Then by the chain rule for conservative field \cite{bolte2021conservative}, we can conclude that $h$ is a potential function that admits $\D_h$ as its conservative field. Moreover, as $\D_f$ and $\partial \ca{P}_+$ are convex-valued over $\Rn$, it holds that $\D_h$ is convex-valued over $\Rn\times \Rn \times \Rn$.   This completes the proof. 
\end{proof}

\begin{prop}
	\label{Prop_Lyapunov_ST}
	Suppose Assumption \ref{Assumption_f} and Assumption \ref{Assumption_framework_ST} hold. Then $h$ is a Lyapunov function for the differential inclusion \eqref{Eq_DI_ST} with the stable set $\{(x,m,v)\in \Rn \times \Rn \times \Rn: 0\in \D_g(x), m +\sigma x = 0\}$. 
\end{prop}
\begin{proof}
	For any trajectory of the differential inclusion \eqref{Eq_DI_ST}, there exists $l_f: \bb{R}_+ \to \Rn$ and $l_u: \bb{R}_+ \to \Rn$ such that $l_f(s) \in \D_f(x(s))$ and $l_u(s) \in \ca{U}(x(s), m(s))$ for almost every $s \geq 0$, and 
	\begin{equation}
		\left( \dot{x}(s), \dot{m}(s), \dot{v}(s) \right) 
		= 
		\left[\begin{matrix}
			-(\ca{P}_+(v(s)) + \varepsilon )^{-\frac{1}{2}} \odot \left(   m(s) + \sigma x(s) \right)\\
			-\tau_1 m(s) + \tau_1 l_f(s)\\
			-\tau_2 \ca{P}_+(v(s)) + \tau_2  l_u(s)\\
		\end{matrix}\right].
	\end{equation}
	
	Then from the formulation of $h$, we have
	\begin{equation*}
		\small
		\begin{aligned}
			&\inner{\D_h(x(s), m(s), v(s)), (\dot{x}(s), \dot{m}(s), \dot{v}(s)) }\\
			={}& -\inner{\D_f(x(s)) + \sigma x(s) + \frac{\sigma}{\tau_1}(\ca{P}_+(v(s))+\varepsilon)^{-\frac{1}{2}} \odot (m(s)+\sigma x(s)), (\ca{P}_+(v(s)) + \varepsilon )^{-\frac{1}{2}} \odot \left(   m(s) + \sigma x(s) \right)} \\
			& + \inner{(\ca{P}_+(v(s)) + \varepsilon )^{-\frac{1}{2}} \odot (m(s) +\sigma x(s)), -m(s) +  l_f(s)} \\
			& + \frac{\tau_2}{4\tau_1} \inner{(m(s) +\sigma x(s))^2 \odot (\ca{P}_+(v(s))+\varepsilon)^{-\frac{3}{2}} \odot \partial \ca{P}_+(v(s)), v(s) - l_u(s)  }\\
			\ni{}& -\inner{l_f(s) + \sigma x(s) + \frac{\sigma}{\tau_1}(\ca{P}_+(v(s))+\varepsilon)^{-\frac{1}{2}} \odot (m(s)+\sigma x(s)), (\ca{P}_+(v(s)) + \varepsilon )^{-\frac{1}{2}} \odot \left(   m(s) + \sigma x(s) \right)} \\
			& + \inner{(\ca{P}_+(v(s)) + \varepsilon )^{-\frac{1}{2}} \odot (m(s) +\sigma x(s)), -m(s) +  l_f(s)} \\
			& + \frac{\tau_2}{4\tau_1} \inner{(m(s) +\sigma x(s))^2 \odot (\ca{P}_+(v(s))+\varepsilon)^{-\frac{3}{2}} \odot \partial \ca{P}_+(v(s)), v(s) - l_u(s)  }\\
			\leq{}& -\frac{\sigma}{\tau_1} \inner{ (\ca{P}_+(v(s))+\varepsilon)^{-1} \odot (m(s)+\sigma x(s)), m(s)+\sigma x(s) }  \\
			& - \inner{ (\ca{P}_+(v(s))+\varepsilon)^{-\frac{1}{2}} \odot (m(s)+\sigma x(s)), m(s)+\sigma x(s) }\\
			& + \frac{\tau_2}{4\tau_1} \inner{(m(s) +\sigma x(s))^2, \ca{P}_+(v(s)) \odot (\ca{P}_+(v(s))+\varepsilon)^{-\frac{3}{2}}  }\\
			\leq{}& -\frac{\sigma}{\tau_1} \inner{ (\ca{P}_+(v(s))+\varepsilon)^{-1} \odot (m(s)+\sigma x(s)), m(s)+\sigma x(s) },
		\end{aligned}
	\end{equation*}
	Here the first inequality follows from the fact that $l_u(s) \geq 0$ and $\partial\ca{P}_+(v) \odot v = \ca{P}_+(v)$. 
	Therefore, we can conclude that for any initial point $(x(0), m(0), v(0)) \in \Rn \times \Rn \times \Rn$, it holds for any $t \geq 0$ that  
	\begin{equation}
		\label{Eq_Prop_Lyapunov_ST_0}
		\begin{aligned}
			&h(x(t), m(t), v(t)) - h(x(0), m(0), v(0)) \\
			\leq{}&  \int_{0}^{t} {\min_{l \in \D_h(x(s), m(s), v(s))}} \inner{l, (\dot{x}(s), \dot{m}(s), \dot{v}(s)) } \mathrm{d}s\\
			\leq{}&  -\frac{\sigma}{\tau_1} \int_{0}^{t}  \inner{ (\ca{P}_+(v(s))+\varepsilon)^{-1} \odot (m(s)+\sigma x(s)), m(s)+\sigma x(s) } \mathrm{d}s. 
		\end{aligned}
	\end{equation}
	As a result, we can conclude that for any  trajectory of the differential inclusion \eqref{Eq_DI_ST}, it holds for any $t > 0$ that $h(x(t), m(t), v(t)) \leq  h(x(0), m(0), v(0))$.
	
	Now consider the case when $(x(0), m(0), v(0)) \notin  \{(x,m,v)\in \Rn \times \Rn \times \Rn: 0\in \D_g(x), m +\sigma x = 0\}$.  Suppose there exists some $T > 0$ such that  
	\begin{equation}
		\label{Eq_Prop_Lyapunov_ST_1}
		h(x(T), m(T), v(T)) =  h(x(0), m(0), v(0)). 
	\end{equation}
	Then \eqref{Eq_Prop_Lyapunov_ST_0} implies that $m(s)+\sigma x(s) = 0$ 
	holds for almost every $s\in [0,T]$. Therefore, $\dot{m}(s) + \sigma \dot{x}(s) =0$ and $\dot{x}(s) = 0$ hold for almost every $s \in [0, T]$. As a result, we have 
	\begin{equation*}
		0 = \dot{m}(s) \in  -\tau_1 m(s) + \tau_1 \D_f(x(s)) = \tau_1 \sigma x(s) + \tau_1 \D_f(x(s)) 
	\end{equation*}
	holds for almost every $s \in [0, T]$. Together with the facts that $(x(t), m(t), v(t))$ is absolutely continuous and $\D_f$ is graph-closed and locally bounded, we have that 
	\begin{equation*}
		m(0) + \sigma x(0) = 0, \quad 0 \in \D_f(x(0)) + \sigma x(0) \;=\D_g(x(0)).
	\end{equation*}
	But the above contradicts the condition that
	$(x(0), m(0), v(0)) \notin  \{(x,m,v): 0\in \D_g(x), m +\sigma x = 0\}$. 
	As a result, we can conclude that for any $T>0$, whenever $(x(0), m(0), v(0)) \notin  \{(x,m,v): 0\in \D_g(x), m +\sigma x = 0\}$, it holds that 
	\begin{equation*}
		h(x(T), m(T), v(T)) <  h(x(0), m(0), v(0)). 
	\end{equation*}
	This completes the proof. 
\end{proof}

In the next proposition, we show that the interpolated process of the sequence $\{(\xk,\mk, \vk)\}$ is a perturbed solution to the differential inclusion 
\eqref{Eq_DI_ST}.

\begin{prop}
	\label{Prop_perturbed_solution_ST}
	Suppose Assumption \ref{Assumption_f} and Assumption \ref{Assumption_framework_ST} hold. Then almost surely,  the interpolated process of $\{(\xk, \mk, \vk)\}$ is a perturbed solution for the differential inclusion \eqref{Eq_DI_ST}. 
\end{prop}
\begin{proof}
	From the uniform boundedness of $\{\mk\}$, $\{\vk\}$ and $\{g_k\}$ in Lemma \ref{Le_UB_mk_gk_ST} and Lemma \ref{Le_UB_W_ST}, and  Assumption \ref{Assumption_framework_ST}(4), we can conclude that $\lim_{k\to \infty} \norm{\mkp - \mk} + \norm{\vkp - \vk} = 0$. Therefore, there exists a sequence of random variables $\{\tau_k\}$ such that  almost surely, $\lim_{k\to \infty}\tau_k = 0$ holds and $\norm{\mkp - \mk} + \norm{\vkp - \vk} \leq \tau_k$. 
	
	Then from the formulation of  framework \eqref{Eq_Framework}, the sequence $\{(\xk, \mk, \vk)\}$ satisfies the following inclusion
	\begin{equation*}
		(\xkp, \mkp,\vkp) \in (\xk, \mk, \vk) - \eta_k \ca{G}^{\tau_k}(\xk, \mk, \vk) - \eta_k (0, -\tau_1 \xi_{k+1}, 0).
	\end{equation*}
	Then it directly follows from  Assumption \ref{Assumption_framework_ST}(4) and Proposition \ref{Prop_UB_martingale_difference_sequence} that 
	\begin{equation*}
		\lim_{s \to \infty} \sup_{s\leq i \leq \Lambda(\lambda_s + T)}\norm{ \sum_{k = s}^{i} \eta_k (0, \tau_1 \xi_{k+1}, 0)} =0. 
	\end{equation*}
	Therefore, we can conclude that the condition (1) and (2) in Lemma \ref{Le_interpolated_process} hold. Moreover, condition (3) in Lemma \ref{Le_interpolated_process} directly follows from Assumption \ref{Assumption_framework_ST}(1), Lemma \ref{Le_UB_mk_gk_ST} and Lemma \ref{Le_UB_W_ST}. Therefore, from Lemma \ref{Le_interpolated_process}, we can conclude that the interpolated process of $\{(\xk, \mk, \vk)\}$ is a perturbed solution for the differential inclusion \eqref{Eq_DI_ST}. This completes the proof.

\end{proof}

In the following theorem, we present the convergence properties of the sequence $\{(\xk, \mk, \vk)\}$, and prove that $\lim_{k\to \infty} \norm{\mk + \sigma \xk} = 0$ almost surely. 
\begin{theo}
	\label{The_convergence_ST}
	Suppose Assumption \ref{Assumption_f} and Assumption \ref{Assumption_framework_ST} hold. Then for the sequence $\{(\xk, \mk, \vk)\}$ generated by the framework \eqref{Eq_Framework}, almost surely, it holds that 
	\begin{enumerate}
		\item any cluster point of the sequence $\{\xk\}$ is a $\D_g$-stationary point of $g$;
		\item $\lim_{k\to \infty} \norm{\mk + \sigma \xk} = 0$;
		\item the sequence of function values $\{g(\xk)\}$ converges.
	\end{enumerate}
\end{theo}
\begin{proof}
	From Proposition \ref{Prop_perturbed_solution_ST}, we can conclude that the interpolated process of $\{(\xk, \mk, \vk)\}$ is a perturbed solution for the differential inclusion \eqref{Eq_DI_ST}.  Moreover, Proposition \ref{Prop_Lyapunov_ST} illustrates that $h$ is a Lyapunov function for the differential inclusion \eqref{Eq_DI_ST} with stable set $\{(x,m,v)\in \Rn \times \Rn \times \Rn: 0\in \D_g(x), m +\sigma x = 0\}$. Then we can conclude that any cluster point of $\{(\xk, \mk, \vk)\}$ lies in the set $\{(x,m,v)\in \Rn \times \Rn \times \Rn: 0\in \D_g(x), m +\sigma x = 0\}$, and the sequence $\{h(\xk, \mk, \vk)\}$ converges. 
	
	As a result, we first conclude that any cluster point of $\{\xk\}$ lies in the set $\{x \in \Rn: 0\in \D_g(x)\}$, and any cluster point of $\{(\xk, \mk)\}$ lies in $\{(x, m) \in \Rn \times\Rn: \sigma x + m  = 0\}$. As a result, noting that $\{\sigma x_k+m_k\}$ is bounded in $\bb{R}^n$, it holds that $\lim_{k\to \infty} \norm{\sigma \xk + \mk} = 0$. Furthermore, notice that 
	\begin{equation*}
		\lim_{k\to \infty} |h(\xk, \mk, \vk) - g(\xk)| \leq \lim_{k\to \infty} \frac{1}{2\tau_1\sqrt{\varepsilon}}\norm{\sigma \xk +\mk}^2 = 0,
	\end{equation*}
	we can deduce that the sequence  $\{g(\xk)\}$ converges. This completes the proof. 
\end{proof}

Theorem \ref{The_convergence_ST} illustrates that $\lim_{k\to \infty} \norm{\xk - \yk} = 0$. 
Therefore, substituting the formulation of $\{\yk\}$ in \eqref{Eq_Defin_yk} into the update scheme of $\{\mk\}$ in framework \eqref{Eq_Framework}, we conclude that $\{\yk\}$ follows the same scheme as \eqref{Eq_update_yk}. Together with the fact that $\lim_{k\to \infty} \norm{\xk - \yk} = 0$, based on  the same proof techniques as in Lemma \ref{Le_approx_dk}, we can conclude that there exists a sequence of non-negative random variables $\{\tau_k\}$ such that $\lim_{k\to \infty}\tau_k = 0$ holds almost surely, and 
\begin{equation*}
	\ykp \in \yk - \frac{\theta_k}{\sigma}(\D_g^{\tau_k}(\yk) + \xi_{k+1}).
\end{equation*}
Then we have the following corollary showing that the interpolated process of the sequence $\{\yk\}$ is a perturbed solution of the differential inclusion \eqref{Eq_DI_SGD_yk}. The proof of Corollary \ref{Coro_perturbed_solution_yk_ST} is the same as Proposition \ref{Prop_perturbed_solution_yk}, hence is omitted for simplicity. 

\begin{coro}
	\label{Coro_perturbed_solution_yk_ST}
	Suppose Assumption \ref{Assumption_f} and Assumption \ref{Assumption_framework_ST} hold. Then the interpolated path of the sequence $\{\yk\}$ is a perturbed solution of the differential inclusion \eqref{Eq_DI_SGD_yk}. 
\end{coro}

\section{Application: Adam with Decoupled Weight Decay}

In this section, we propose a novel variant of Adam method, which is named as Adam with decoupled weight decay (AdamD). As an application of our theoretical analysis in Section 3 and Section 4, we show the convergence properties of AdamD directly from the results in Theorem \ref{The_convergence_Nondiminishing} and Theorem \ref{The_convergence_ST}. 

Throughout this section, we focus on the settings where $f$ in \eqref{Prob_Ori} takes the following finite-sum formulation:
\begin{equation}
	\label{Eq_f_finite_sum}
	f(x) = \frac{1}{N}\sum_{i = 1}^N f_i(x).
\end{equation}
Here we make the following assumptions on the functions $\{f_i: i \in [N]\}$ in \eqref{Eq_f_finite_sum}. 
\begin{assumpt}
	\label{Assumption_implementation_f}
	\begin{enumerate}
		\item For each $i \in [N]$, $f_i$ is a definable function that admits a definable set-valued mapping $\D_{f_i}$ as its conservative field.
		\item $\sup_{i \in [N], ~ x \in \Rn} \norm{\D_{f_i}(x)} <\infty$. 
		\item $f$ is bounded from below. 
	\end{enumerate}
\end{assumpt}

As demonstrated in \cite{bolte2021conservative}, for any neural network that is built from definable blocks, the conservative field corresponds the AD algorithms is a definable set-valued mapping. Hence, we can conclude that Assumption \ref{Assumption_implementation_f}(1) can be satisfied in a wide range of training tasks. Moreover, Assumption \ref{Assumption_implementation_f}(2) assumes the Lipschitz continuity of the function $f$, which is common in various existing works \cite{barakat2021convergence,guo2021novel,shi2021rmsprop,zhang2022adam}. 

Moreover, \cite[Corollary 4]{bolte2021nonsmooth} illustrates that $f$ is a path-differentiable function and admits $\frac{1}{N}\sum_{i = 1}^N \D_{f_i}$ as its conservative field. Therefore, in the rest of this section, we choose the conservative field $\D_f$ as
\begin{equation}
	\D_f(x) = \conv\left(\frac{1}{N}\sum_{i = 1}^N \D_{f_i}(x)  \right). 
\end{equation}
The detailed AdamD method is presented in Algorithm \ref{Alg:ADAM}. 
In our proposed AdamD method, the weight decay term $\sigma \xk$ is decoupled from the update schemes for $\{\mk\}$ and $\{\vk\}$. In particular,  the estimators $\{\vk\}$ are updated as an exponential moving average over $\{g_k^2\}$ with parameter $\beta \in (0,1)$. 
\begin{algorithm}[htbp]
	\begin{algorithmic}[1]  
		\Require Initial point $x_0 \in \Rn$, $m_0 \in \Rn$ and $v_0 \in \Rn_+$,  weight decay parameter $\sigma> 0$, safeguard parameter $\varepsilon > 0$, stepsizes $\eta \leq \frac{1}{\sigma \varepsilon}$ and $\beta \in (0, 1)$;
		\State Set $k=0$;
		\While{not terminated}
		\State Independently sample $i_k$ from $[N]$, and compute $g_k \in \D_{f_{i_k}}(\xk)$;
		\State Update the momentum term by $m_{k+1} = (1-  \theta_k) m_k + \theta_k   g_k$;
		\State Update the estimator $\vkp$ by $\vkp = (1-\beta) \vk + \beta g_k^2$;
		\State Update $\xk$ by $\xkp = \xk - \eta (\sqrt{\vkp} + \varepsilon )^{-1} \odot (  m_{k+1} +  \sigma \xk)$;
		\State $k = k+1$;
		\EndWhile
		\State Return $\xk$.
	\end{algorithmic}  
	\caption{Adam with decoupled weight decay (AdamD) for nonsmooth optimization problem \eqref{Prob_Ori}. }  
	\label{Alg:ADAM}
\end{algorithm}

Then based on the convergence properties of the framework \eqref{Eq_Framework} presented in  Theorem \ref{The_convergence_Nondiminishing}, the following theorem illustrates the convergence properties of Algorithm \ref{Alg:ADAM} with non-diminishing $\{\eta_k\}$. 

\begin{theo}
	\label{Theo_implementation_convergence_nondiminishing}
	Suppose Assumption \ref{Assumption_implementation_f} holds. Moreover, we assume that the sequence of stepsizes $\{\theta_k\}$ is a positive sequence that satisfies $\sum_{k=0}^{\infty} \theta_k = \infty$, $\lim_{k\to \infty} \theta_k \log(k) = 0$. Then almost surely, any cluster point of $\{\xk\}$ in Algorithm \ref{Alg:ADAM} is a $\D_g$-stationary point of $g$, and the sequence $\{g(\xk)\}$ converges. 
\end{theo}
\begin{proof}
	We first verify the validity of Assumption \ref{Assumption_f}. The definability of $f_i$ and $\D_{f_i}$ implies the definability of $f$ and $\D_f$, hence from \cite[Theorem 5]{bolte2021conservative}, $f$ is path-differentiable and the set $\{f(x): 0\in \D_f(x)\}$ is a finite subset of $\bb{R}$. This verifies the validity of Assumption \ref{Assumption_f}. 
	
	Moreover, let $\{\ca{F}_k\}$ be a sequence of $\sigma$-fields generated by $\{x_j, d_j,m_j: j\leq k\}$,  $d_k = \bb{E}[g_k|\ca{F}_k]$ and $\xi_{k+1} = g_k - d_k$. Then we can conclude that $d_k \in \D_f(\xk)$ and $\bb{E}[\xi_{k+1}|\ca{F}_k] = 0$. Moreover, Assumption \ref{Assumption_implementation_f}(2) illustrates that there exists a constant $M_f$ such that $\sup_{i \in [N], ~x \in \Rn} \norm{\D_f(x)} \leq M_f$. Thus we can conclude that $\sup_{k\geq 0}\norm{g_k} \leq M_f$ and $\sup_{k\geq 0}\norm{d_k} \leq M_f$ hold almost surely. Then $\sup_{k\geq 0}\norm{\xi_{k+1}} \leq 2M_f$    holds almost surely. This verifies the validity of Assumption \ref{Assumption_framework}(5). 
	
	Furthermore, from the update scheme in Step 5 of Algorithm \ref{Alg:ADAM}, we can conclude that  $\sup_{k\geq 0}\norm{\vk} \leq \sup_{k\geq 0} \norm{g_k^2} \leq M_f^2$. 
	This illustrates that Assumption \ref{Assumption_framework}(1) holds with $\varepsilon_v = \frac{1}{M_f+\varepsilon}$ and $M_v = \frac{1}{\varepsilon}$. In addition, Proposition \ref{Prop_UB} directly guarantees that Assumption \ref{Assumption_framework}(2) holds. Assumption \ref{Assumption_framework}(3) follows from the conditions on $\{\eta_k\}$ and $\{\theta_k\}$, while Assumption \ref{Assumption_framework}(4) is implied by the fact that $d_k = \bb{E}[g_k|\ca{F}_k] \in \D_f(\xk)$. Therefore, from Theorem \ref{The_convergence_Nondiminishing}, we can conclude that any cluster point of  the sequence $\{\xk\}$ is a $\D_g$-stationary point of $g$, and the sequence $\{g(\xk)\}$ converges. This completes the proof. 
	
\end{proof}

In the following theorem, we establish the convergence properties for Algorithm \ref{Alg:ADAM} when it is equipped with single-timescale stepsizes. 
The results in Theorem \ref{Theo_implementation_convergence_ST} are direct consequences of Theorem \ref{The_convergence_ST}, hence we omit the proof  for simplicity. 
\begin{theo}
	\label{Theo_implementation_convergence_ST}
	Suppose Assumption \ref{Assumption_implementation_f} holds. Moreover, we assume that 
	\begin{enumerate}
		\item The stepsizes $\eta$ and $\beta$ are replaced by $\eta_k $ and $\beta_k$ respectively in Algorithm \ref{Alg:ADAM};
		\item There exists constants $\tau_2 \geq 4\tau_1 > 0$ such that $\theta_k = \tau_1 \eta_k$ and $\beta_k = \tau_2 \eta_k$ hold for any $k\geq 0$. Moreover, the sequence $\{\eta_k\}$ satisfies $\sum_{k=0}^{\infty} \eta_k = \infty$ and $\lim_{k\to \infty} \eta_k \log(k) = 0$. 
		\item In Step 6 of Algorithm \ref{Alg:ADAM}, the sequence $\{\xk\}$ follows the update scheme
		\begin{equation*}
			\xkp = \xk - \eta_k (\vkp + \varepsilon )^{-\frac{1}{2}} \odot (  m_{k+1} +  \sigma \xk).
		\end{equation*}
	\end{enumerate}
	Then almost surely, any cluster point of $\{\xk\}$ in Algorithm \ref{Alg:ADAM} is a $\D_g$-stationary point of $g$, and the sequence $\{g(\xk)\}$ converges. 
\end{theo}

\section{Numerical Experiments}
In this section, we conduct numerical experiments to demonstrate the effectiveness of AdamD in the context of image classification and language modeling tasks. We compare AdamD with the most popular adaptive algorithms used for training neural networks, i.e. Adam and AdamW. All experiments are conducted using an NVIDIA RTX 3090 GPU and were implemented in Python 3.9 with PyTorch 1.12.0.

\subsection{Implementations of AdamD}

In our numerical experiments, we focus on two key tasks: image classification employing Convolutional Neural Networks (CNNs) and language modeling using Long Short-Term Memory (LSTM) networks \cite{hochreiter1997long}. Specifically, our image classification experiments include the deployment of well-established architectures, namely  Resnet34 \cite{he2016deep} and Densenet121 \cite{huang2018condensenet}, to train the CIFAR-10 and CIFAR-100 datasets \cite{krizhevsky2009learning}. Our language modeling experiments focus on LSTM networks applied to the Penn Treebank dataset \cite{marcus1993building}. It is worth noting that AdamW typically demonstrates superior generalization performance when used to train CNNs for image classification tasks. For training LSTMs, prior studies such as \cite{ding2023nonconvex,loshchilov2017decoupled,zhuang2020adabelief} have observed that Adam exhibits better generalization capacity than AdamW.


\subsubsection{CNNs on image classification}\label{subsubsec:CNN experiments}

In all our experiments on image classification, we train the models consistently for 200 epochs, employing a batch size of 128. At the 150th epoch, we reduce the step size by a factor of 0.1. This step size reduction schedule is a prevalent practice in contemporary deep neural network training.  It is helpful to accelerate the convergence of the optimization algorithm, and to enhance generalization capacity. Similar strategies can be observed in previous works, such as \cite{he2016deep,zhuang2020adabelief}. The weight decay parameter $\sigma$ is fixed to be $5\times10^{-3}$. We use the following hyperparameters setting for tested algorithms:
\begin{itemize}
	\item Adam/AdamW: We search the stepszie $\eta$ within the range of $\{5\times10^{-4},10^{-3},5\times10^{-3},10^{-2},5\times10^{-2},10^{-1},5\times10^{-1},1\}$. Additionally, we set $\varepsilon=10^{-8}$, $\theta_k=10^{-1}$ and $\rho_k=10^{-4}$ as the default setting in Pytorch.
	\item AdamD: We adopt the searching scheme for stepsize as $0.1\times \{5\times10^{-4},10^{-3},5\times10^{-3},10^{-2},5\times10^{-2},10^{-1},5\times10^{-1},1\}$. We set $\theta_s=\frac{\theta_0}{(\log (s+2))^{\frac{3}{2}}}$, with $s$ representing the epoch number. Within the $s$-th epoch, $\theta_k$ {takes the constant value}
	$\theta_s$. Under this setting, we can easily verify that $\theta_k=o(\frac{1}{\log k})$. Here, we set the initial momentum parameter to $\theta_0=10^{-1}$, the second moment parameter to $\beta=10^{-4}$ and the regularization parameter to $\varepsilon=10^{-8}$,
	{which are the same} as the default settings in PyTorch  for Adam/AdamW. 
\end{itemize}

\begin{figure*}[th]
	\begin{center}
		\setlength{\tabcolsep}{0.0pt}  
		\scalebox{1}{\begin{tabular}{cccc}
				\includegraphics[width=0.25\linewidth]{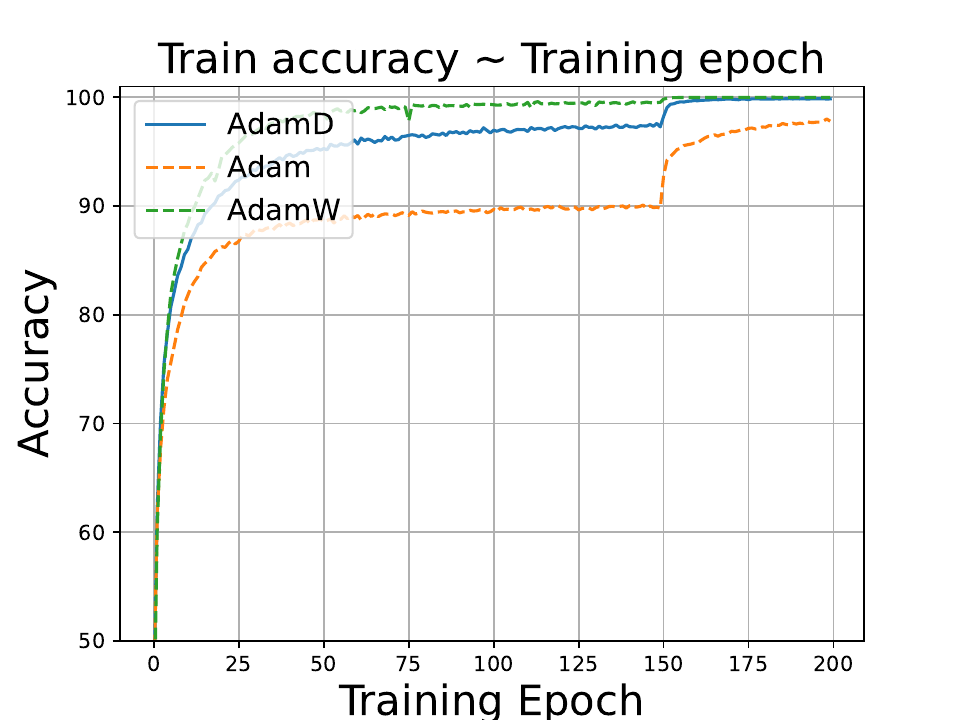}&
				\includegraphics[width=0.25\linewidth]{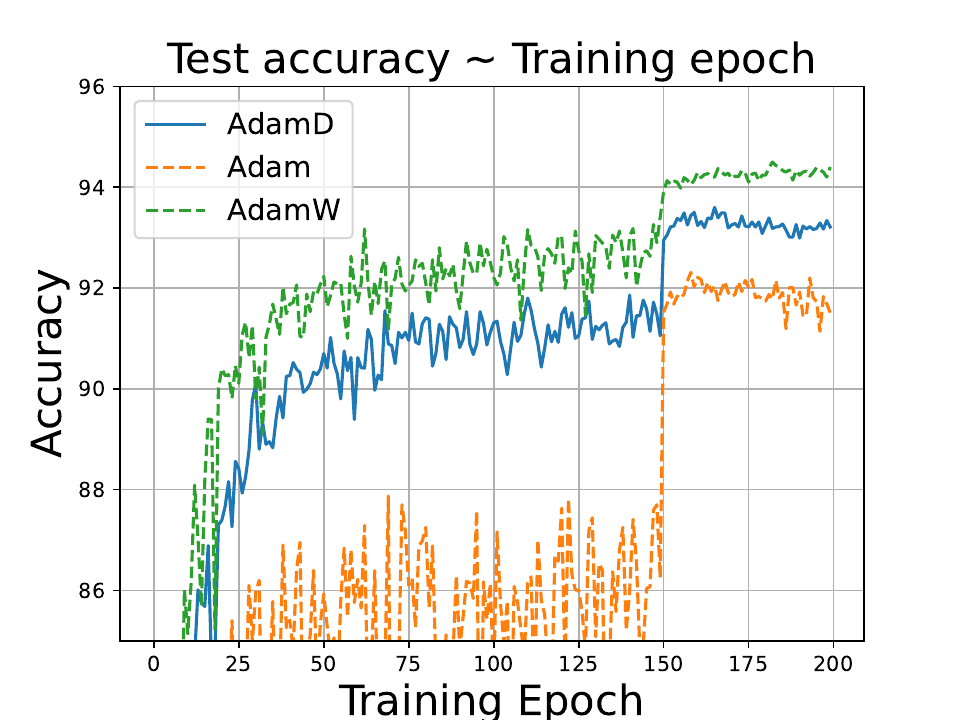}&
				\includegraphics[width=0.25\linewidth]{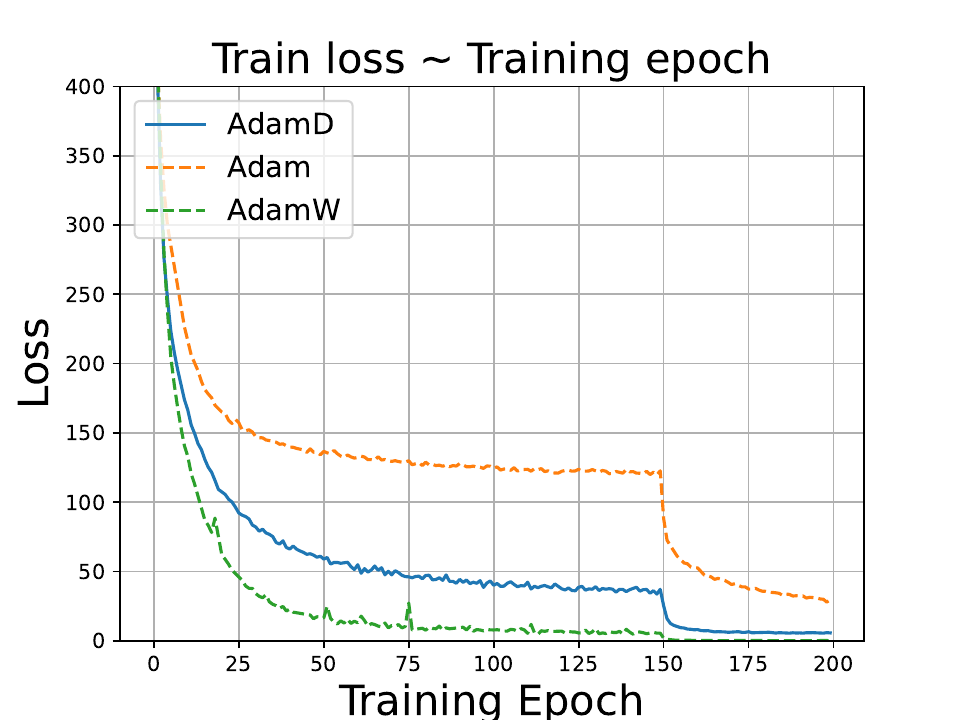} &
				\includegraphics[width=0.25\linewidth]{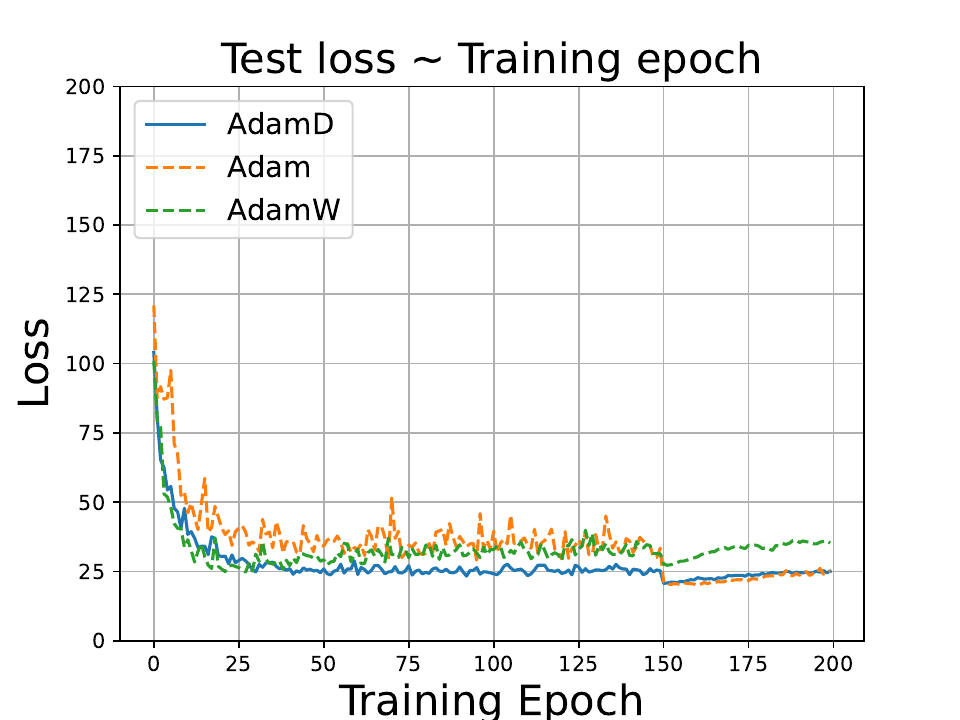}\\
				\multicolumn{1}{c}{\footnotesize{(a) Train accuracy}} &  \multicolumn{1}{c}{\footnotesize{(b) Test accuracy}}&
				\multicolumn{1}{c}{\footnotesize{(c) Train loss}}&
				\multicolumn{1}{c}{\footnotesize{(d) Test loss}}                  
		\end{tabular}}
	\end{center}
	\caption{ResNet34 on CIFAR10 dataset. Stepsize is reduced to 0.1 times of the original value at the 150th epoch.} \label{fig:resnet-cifar10}
\end{figure*}

\begin{figure*}[th]
	\begin{center}
		\setlength{\tabcolsep}{0.0pt}  
		\scalebox{1}{\begin{tabular}{cccc}
				\includegraphics[width=0.25\linewidth]{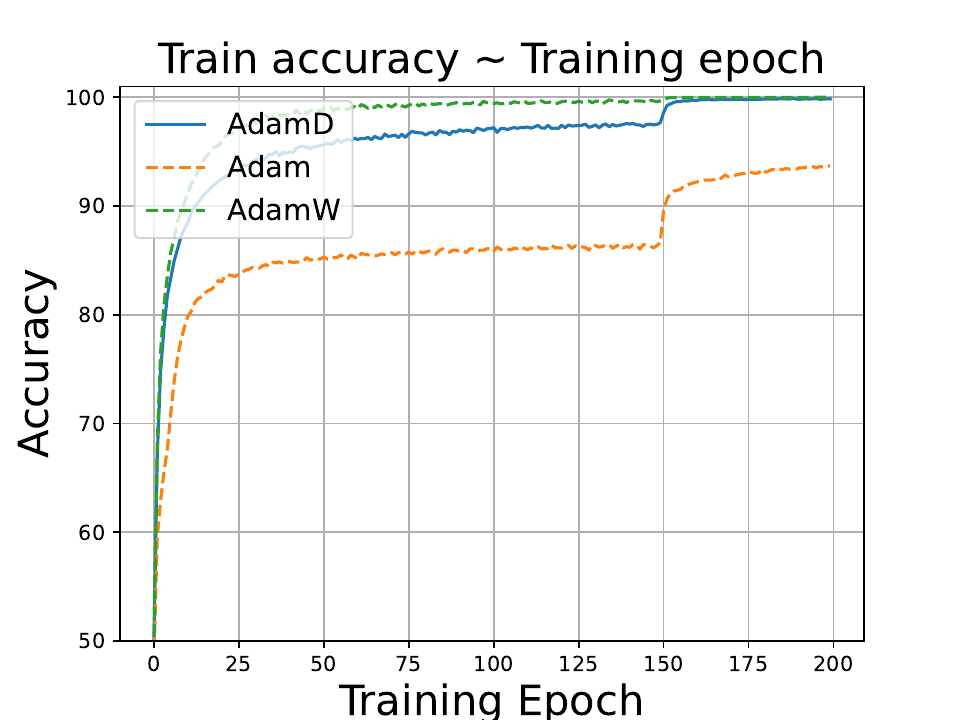}&
				\includegraphics[width=0.25\linewidth]{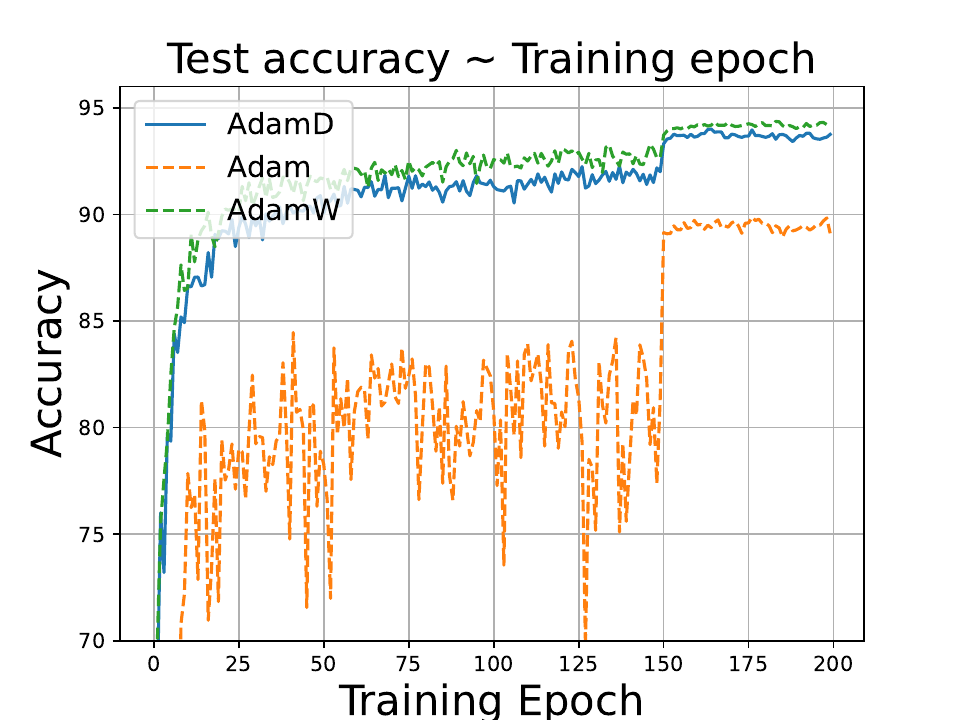}&
				\includegraphics[width=0.25\linewidth]{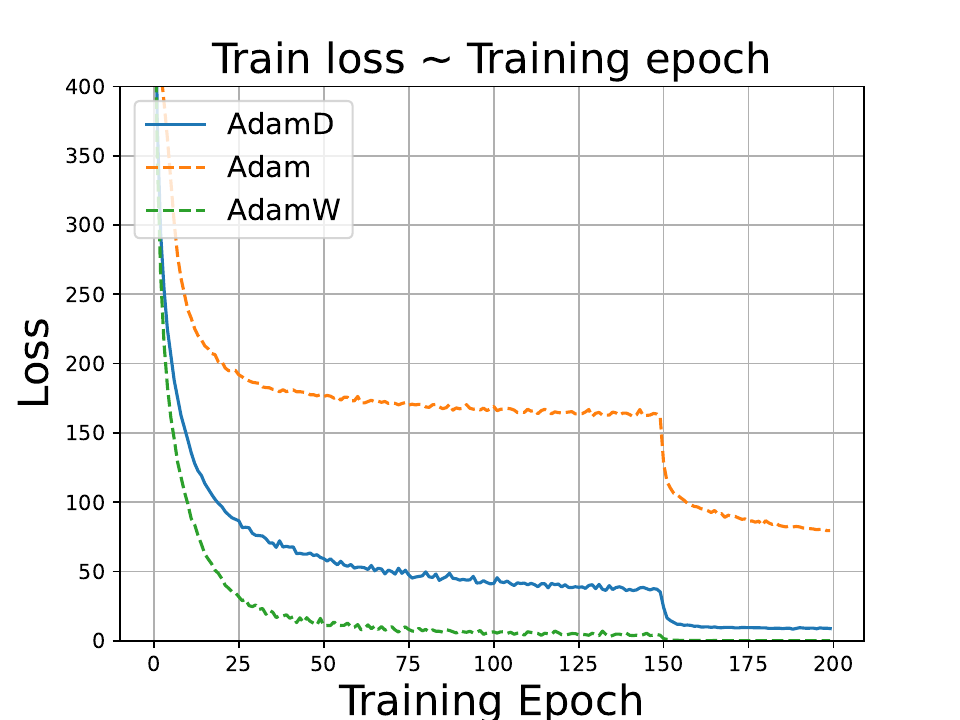} &
				\includegraphics[width=0.25\linewidth]{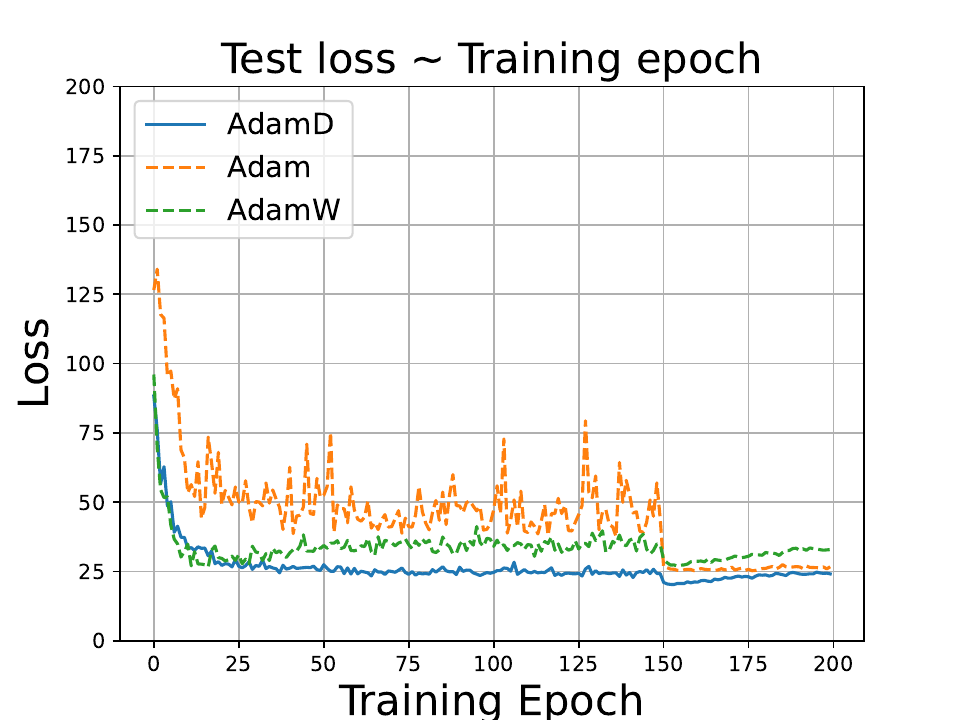}\\
				\multicolumn{1}{c}{\footnotesize{(a) Train accuracy}} &  \multicolumn{1}{c}{\footnotesize{(b) Test accuracy}}&
				\multicolumn{1}{c}{\footnotesize{(c) Train loss}}&
				\multicolumn{1}{c}{\footnotesize{(d) Test loss}}                  
		\end{tabular}}
	\end{center}
	\caption{DenseNet121 on CIFAR10 dataset. Stepsize is reduced to 0.1 times of the original value at the 150th epoch.} \label{fig:densenet-cifar10}
\end{figure*}

\begin{figure*}[th]
	\begin{center}
		\setlength{\tabcolsep}{0.0pt}  
		\scalebox{1}{\begin{tabular}{cccc}
				\includegraphics[width=0.25\linewidth]{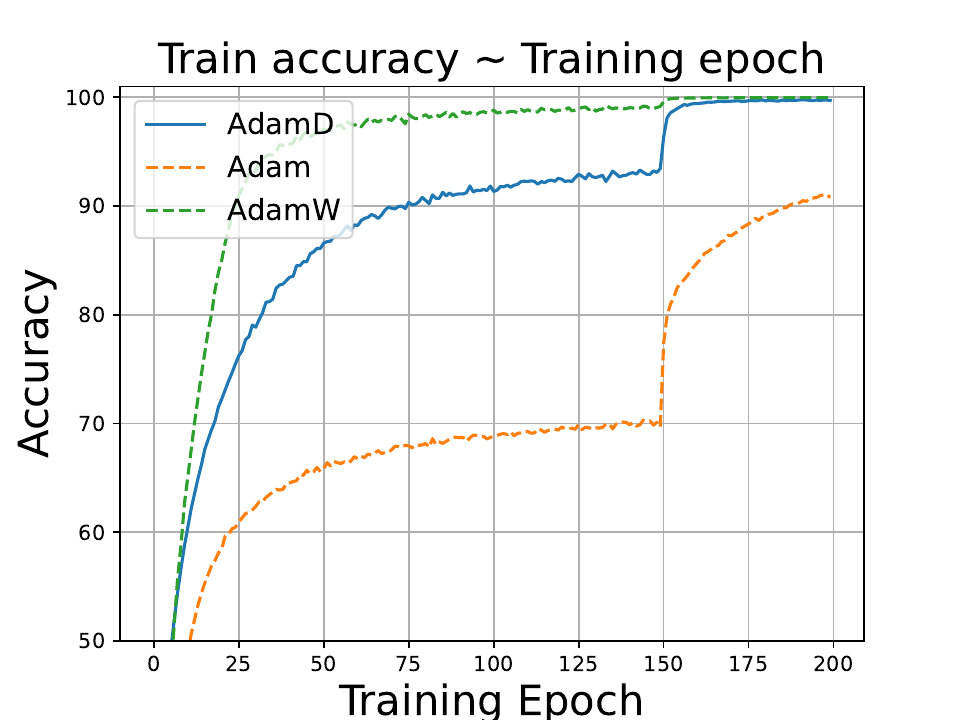}&
				\includegraphics[width=0.25\linewidth]{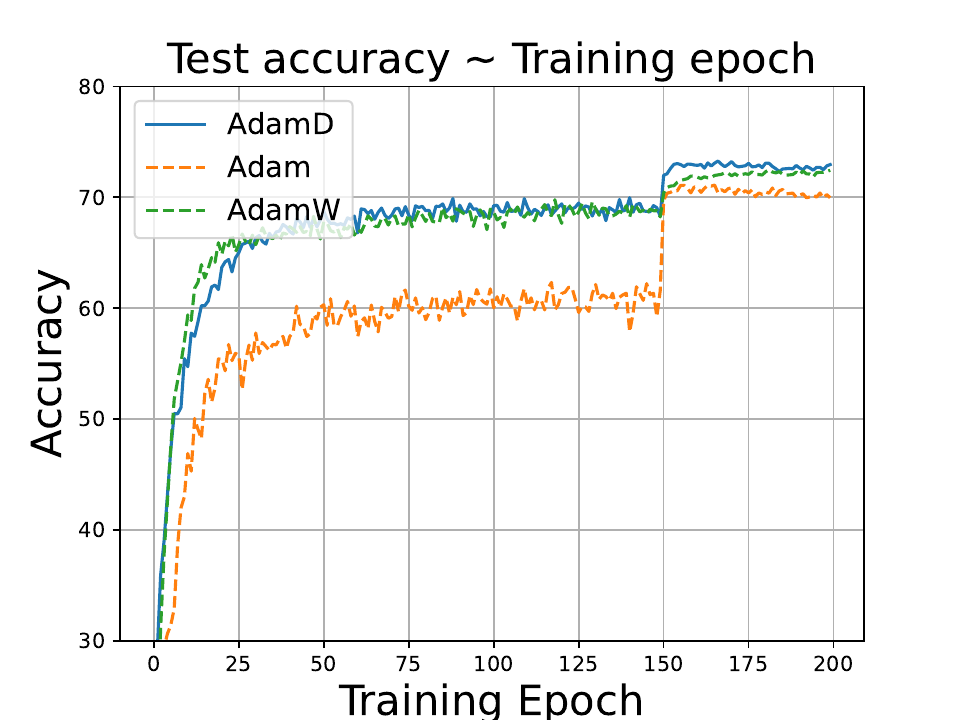}&
				\includegraphics[width=0.25\linewidth]{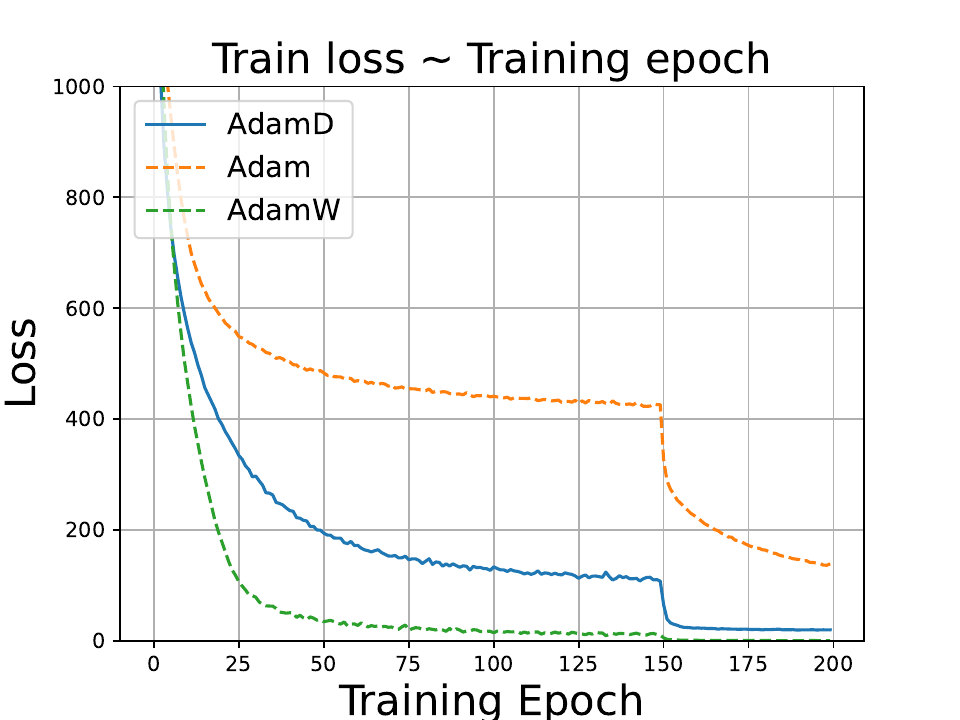} &
				\includegraphics[width=0.25\linewidth]{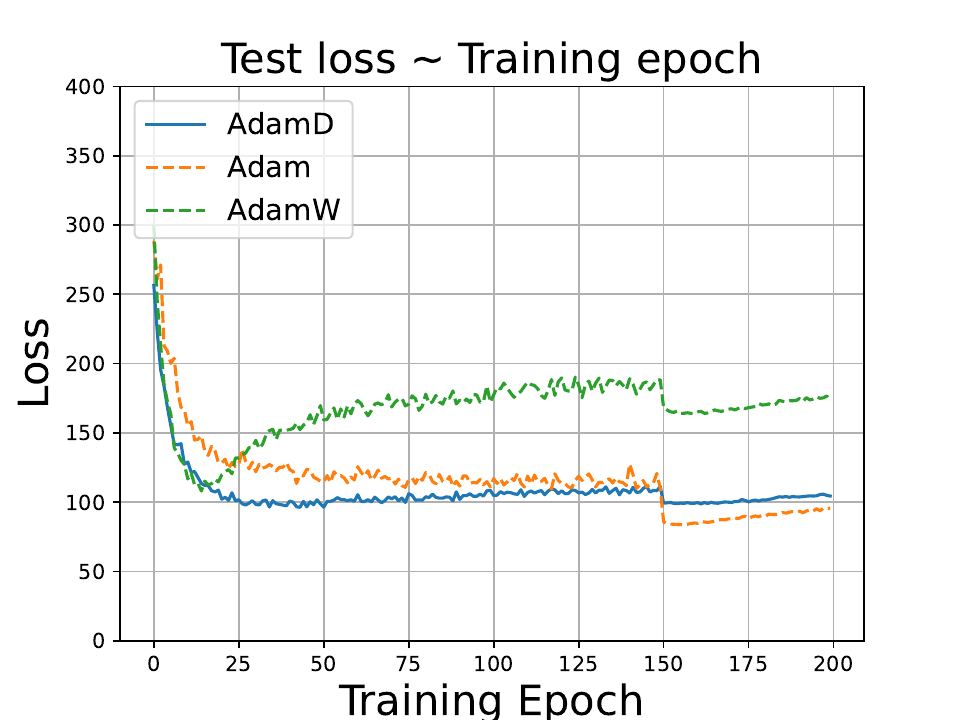}\\
				\multicolumn{1}{c}{\footnotesize{(a) Train accuracy}} &  \multicolumn{1}{c}{\footnotesize{(b) Test accuracy}}&
				\multicolumn{1}{c}{\footnotesize{(c) Train loss}}&
				\multicolumn{1}{c}{\footnotesize{(d) Test loss}}                  
		\end{tabular}}
	\end{center}
	\caption{ResNet34 on CIFAR100 dataset. Stepsize is reduced to 0.1 times of the original value at the 150th epoch.} \label{fig:resnet-cifar100}
\end{figure*}

\begin{figure*}[th]
	\begin{center}
		\setlength{\tabcolsep}{0.0pt}  
		\scalebox{1}{\begin{tabular}{cccc}
				\includegraphics[width=0.25\linewidth]{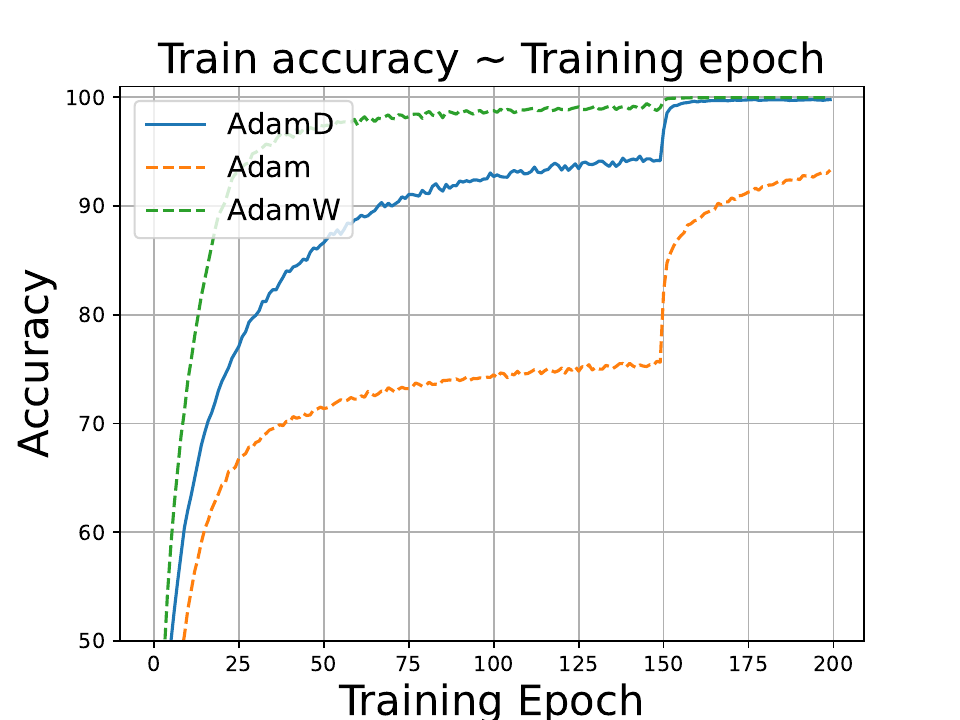}&
				\includegraphics[width=0.25\linewidth]{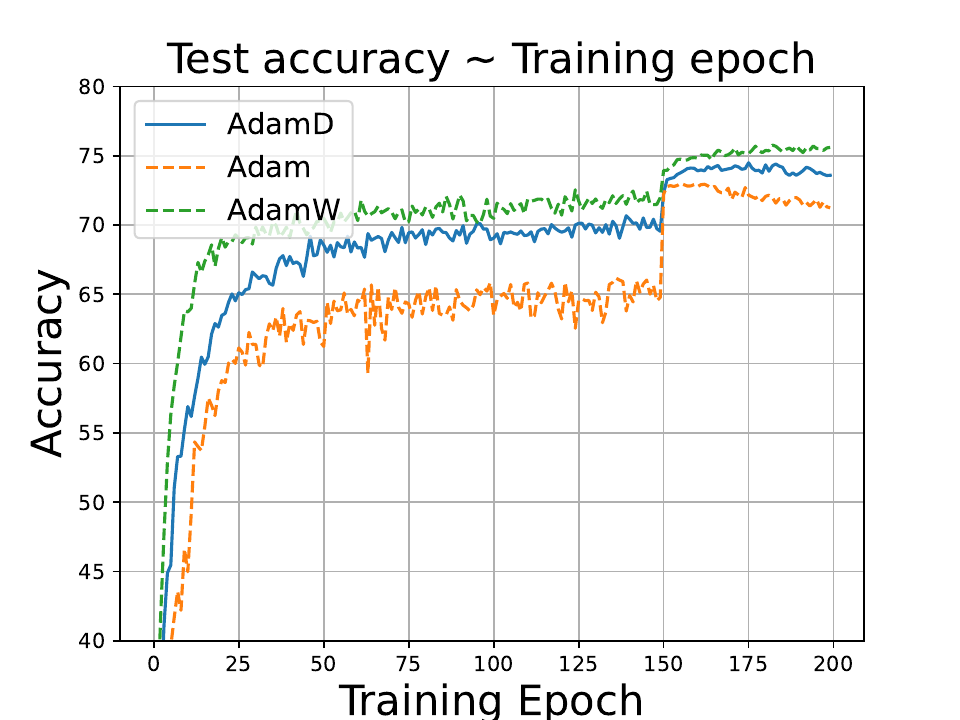}&
				\includegraphics[width=0.25\linewidth]{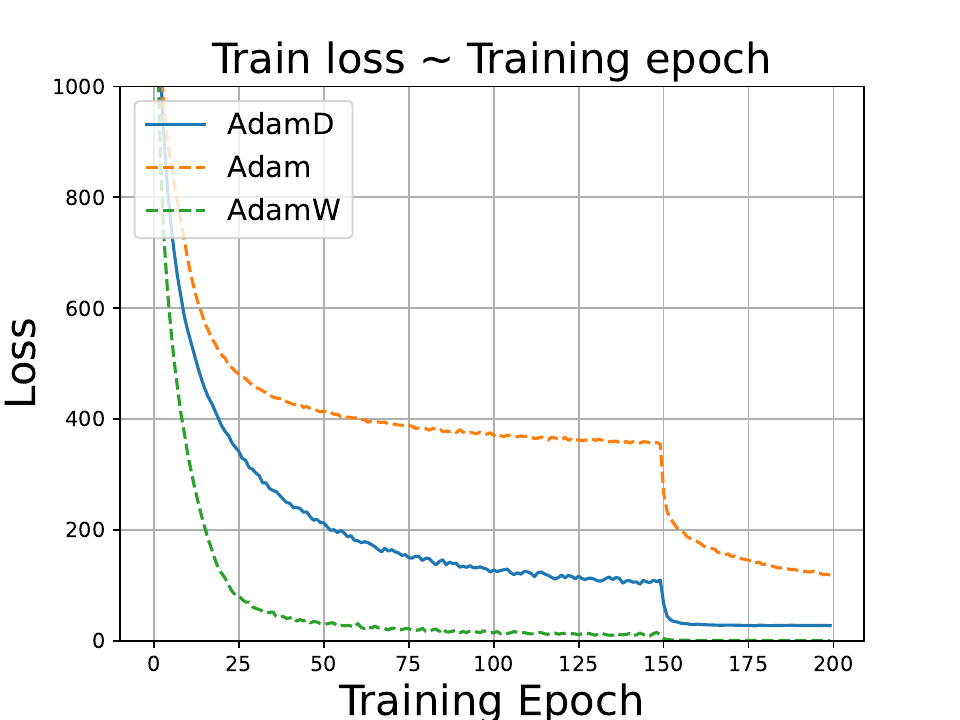} &
				\includegraphics[width=0.25\linewidth]{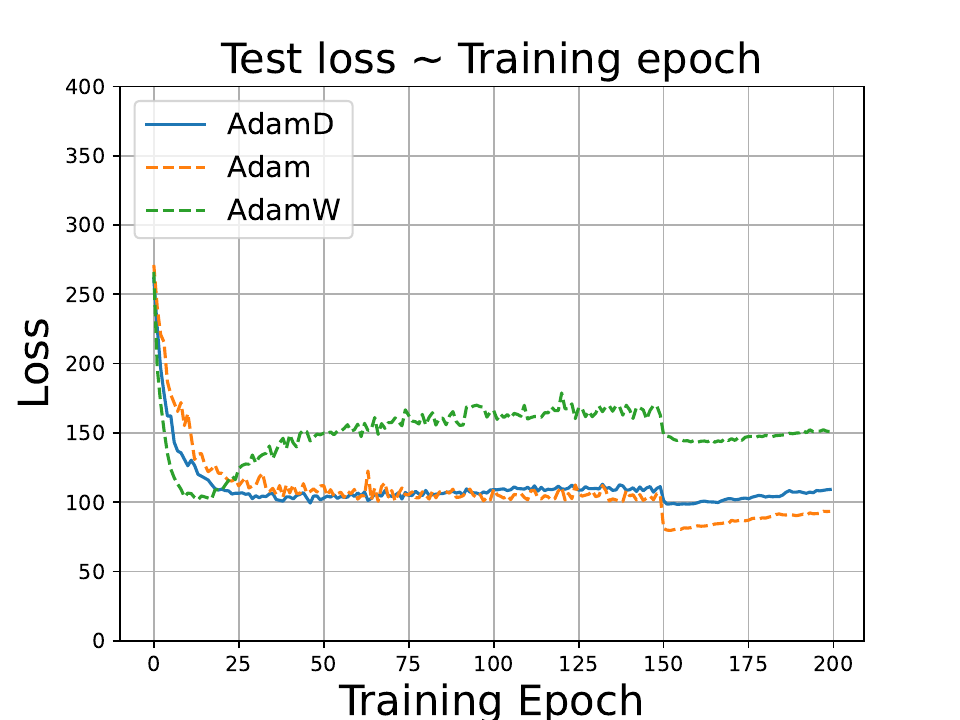}\\
				\multicolumn{1}{c}{\footnotesize{(a) Train accuracy}} &  \multicolumn{1}{c}{\footnotesize{(b) Test accuracy}}&
				\multicolumn{1}{c}{\footnotesize{(c) Train loss}}&
				\multicolumn{1}{c}{\footnotesize{(d) Test loss}}                  
		\end{tabular}}
	\end{center}
	\caption{DenseNet121 on CIFAR100 dataset. Stepsize is reduced to 0.1 times of the original value at the 150th epoch.} \label{fig:densenet-cifar100}
\end{figure*}


\begin{figure*}[th]
	\begin{center}
		\setlength{\tabcolsep}{0.0pt}  
		\scalebox{1}{\begin{tabular}{cccc}
				
				\includegraphics[width=0.33\linewidth]{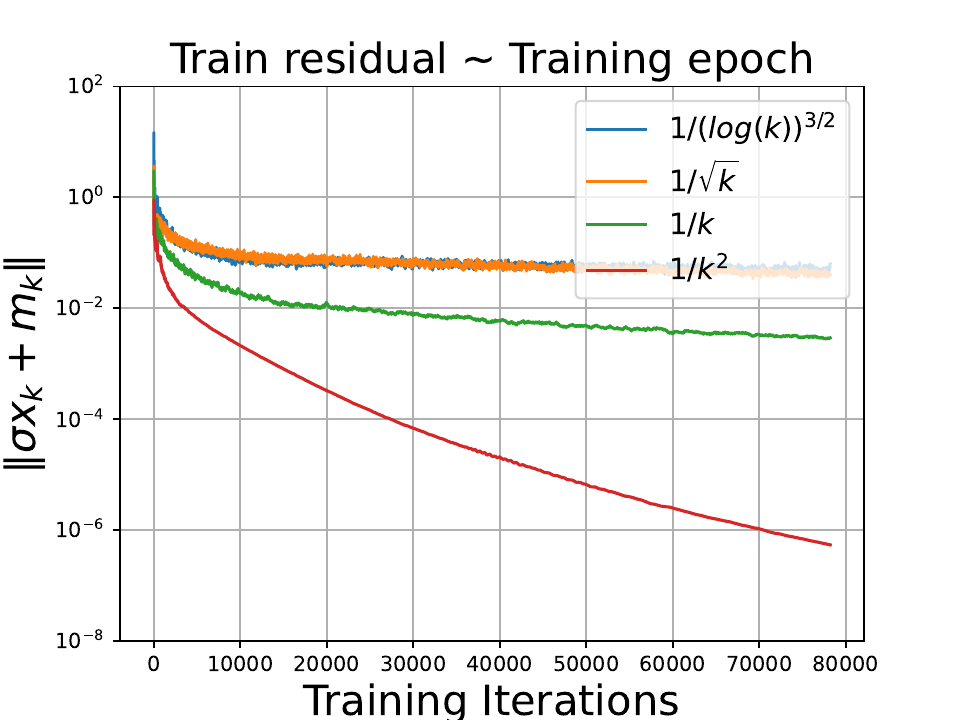}&
				\includegraphics[width=0.33\linewidth]{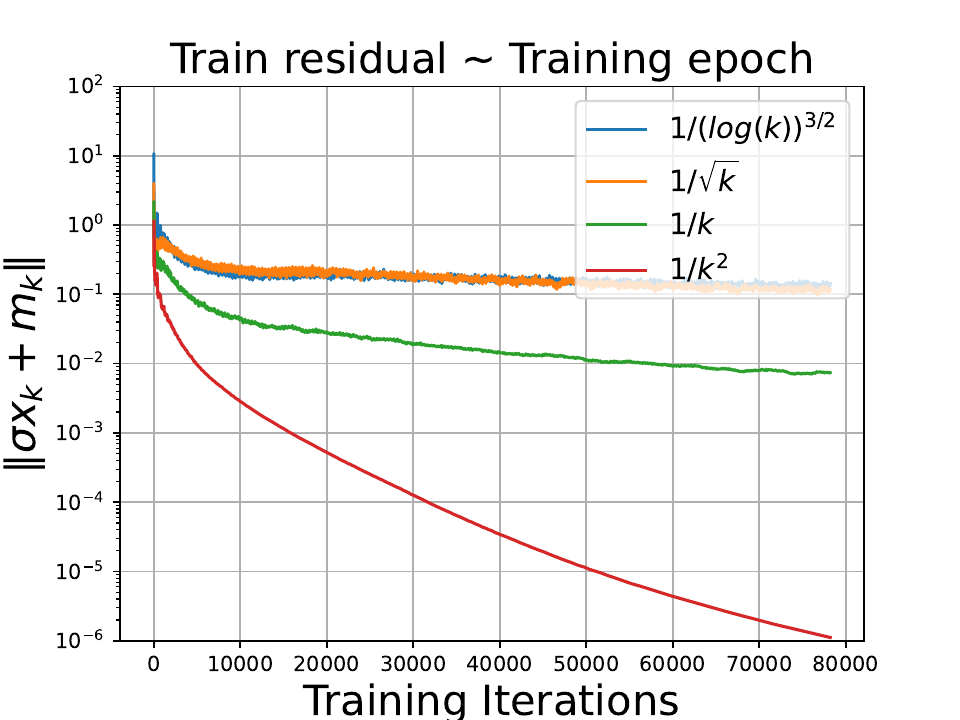} &
				\\
				\multicolumn{1}{c}{\footnotesize{(a) ResNet34 on CIFAR10}}&
				\multicolumn{1}{c}{\footnotesize{(b) ResNet34 on CIFAR100}}                  
		\end{tabular}}
	\end{center}
	\caption{{$\|m_k+\sigma x_k\|$ under different {decay} rates of $\{\theta_k\}$. The stepsizes for updating $\{x_k\}$ are fixed.}} \label{fig:residual}
\end{figure*}

\begin{figure*}[th]
	\begin{center}
		\setlength{\tabcolsep}{0.0pt}  
		\scalebox{1}{\begin{tabular}{ccc}
				\includegraphics[width=0.33\linewidth]{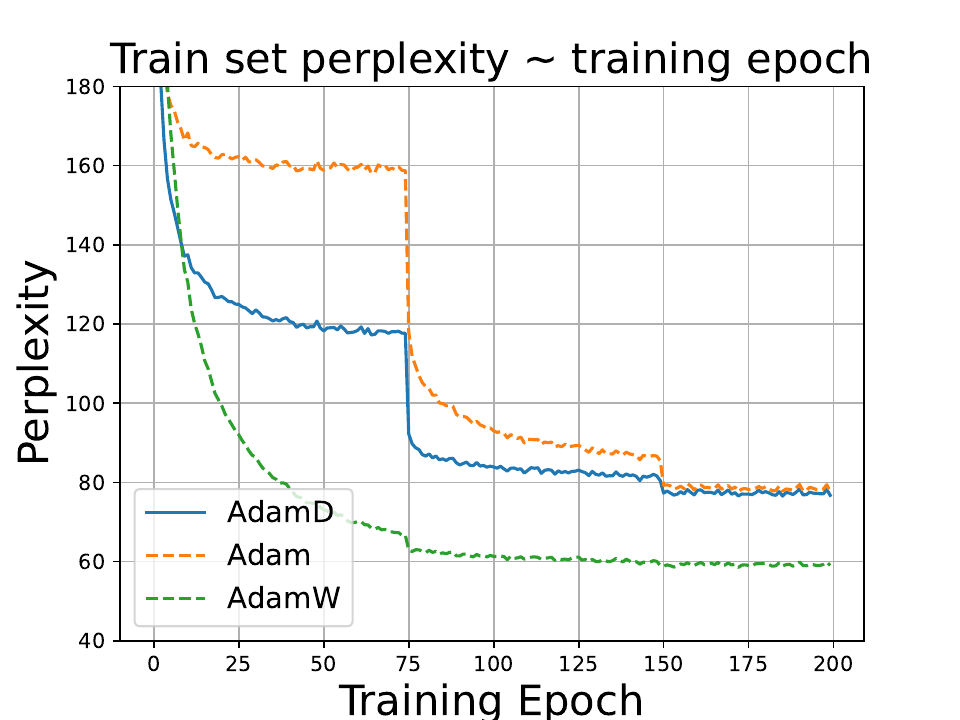}&
				\includegraphics[width=0.33\linewidth]{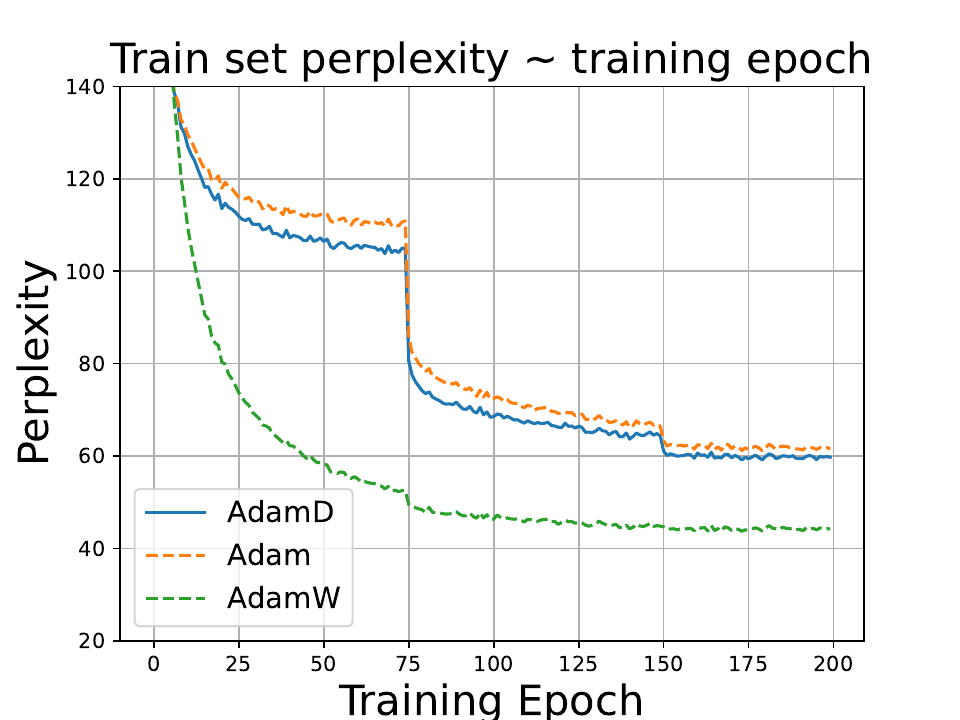}&
				\includegraphics[width=0.33\linewidth]{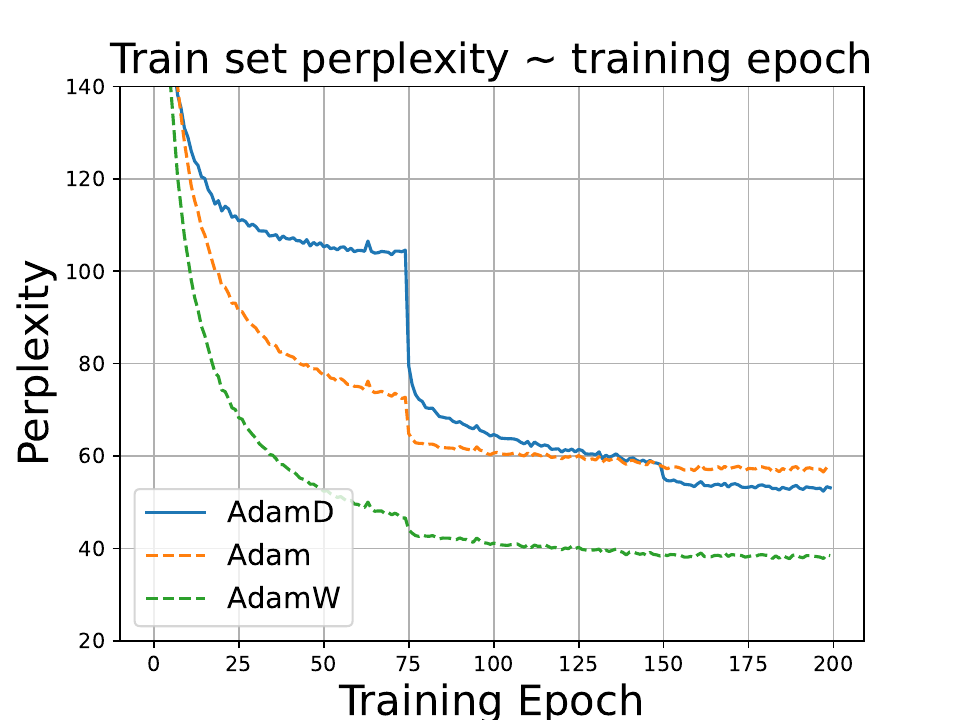}\\
				
				\includegraphics[width=0.33\linewidth]{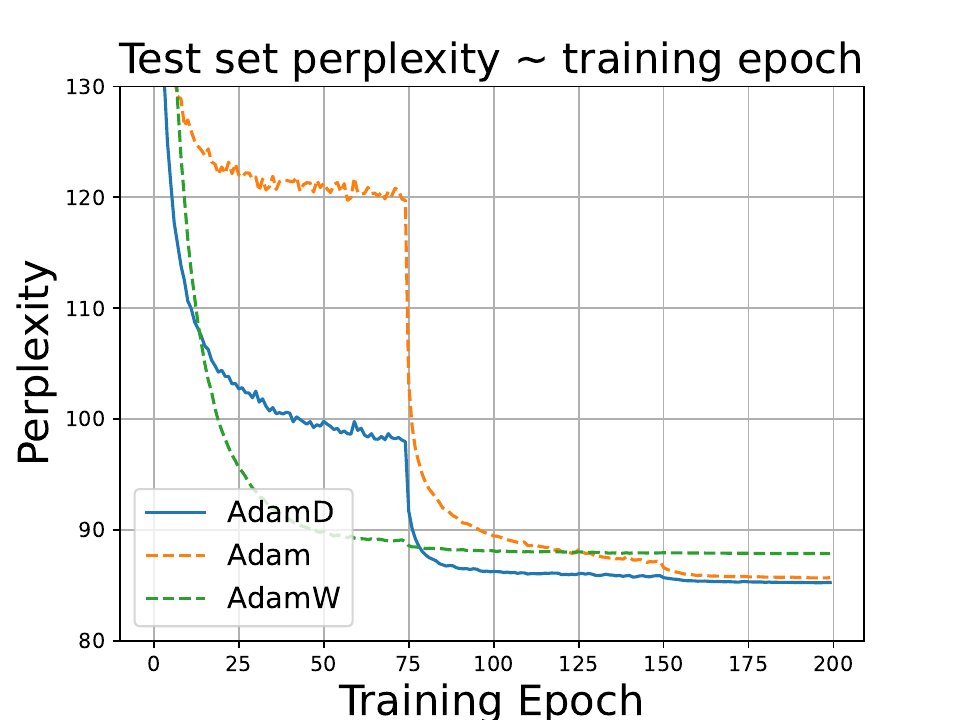}&
				\includegraphics[width=0.33\linewidth]{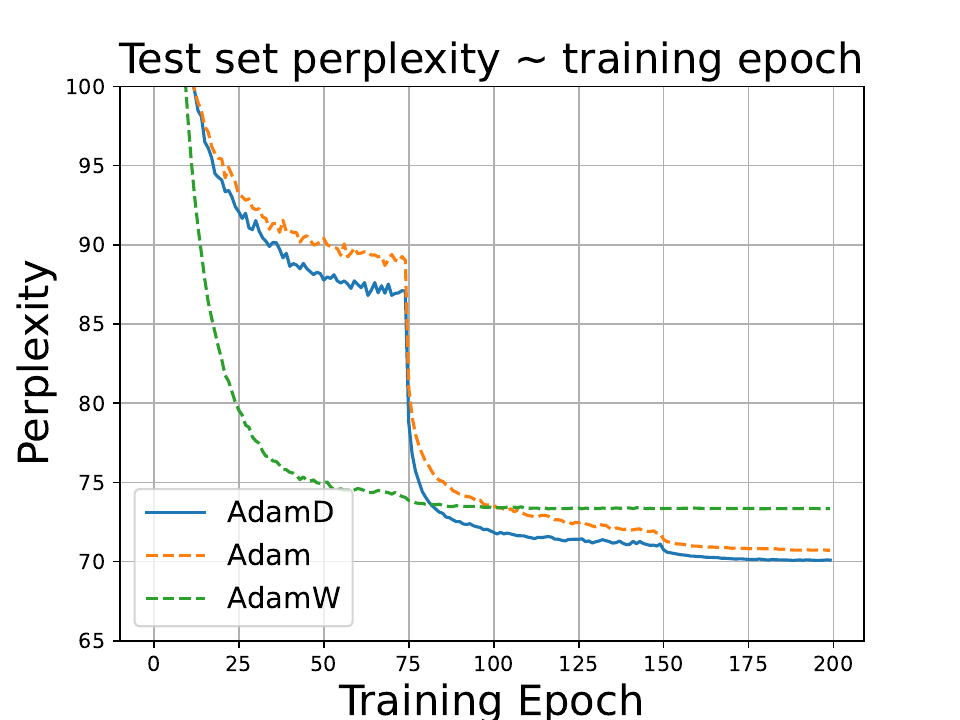}&
				\includegraphics[width=0.33\linewidth]{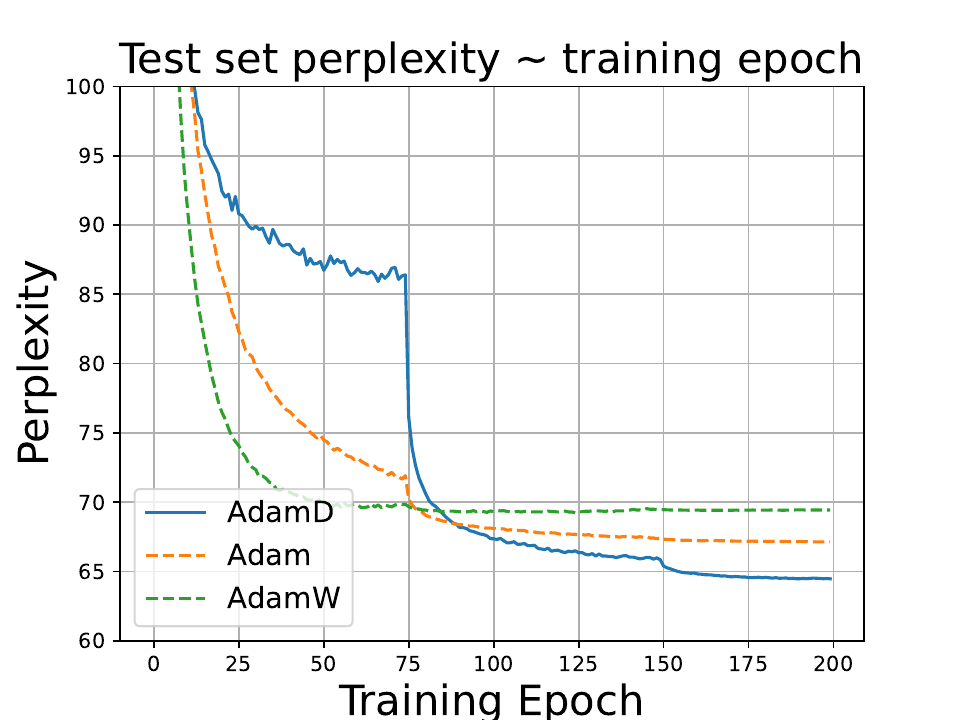}\\
				\footnotesize{(a) 1-layer LSTM} &  \footnotesize{(b) 2-layer LSTM}
				&  \footnotesize{(c) 3-layer LSTM}
				\\
		\end{tabular}}
	\end{center}
	\caption{Training and test perplexity (lower is better) of LSTMs on Penn Treebank dataset with stepsize reduced to 0.1 times of the original value at the 75th epoch and 150th epoch.} \label{fig:lstm}
\end{figure*}

In Step 6 of Algorithm \ref{Alg:ADAM}, the coefficient associated with $x_k$ is expressed as $1-\eta\sigma(\sqrt{v_{k+1}}+\varepsilon)^{-1}$. It is worth noting that as training progresses, the value of $\sqrt{v_{k+1}}+\varepsilon$ tends to become small. To ensure that the coefficient does not become excessively small, in practice, AdamD employs a smaller step size compared to Adam and AdamW. This phenomenon of selecting a smaller scale stepsize also occurs in other optimizers, such as Lion \cite{chen2023symbolic}. The numerical results, as illustrated in Figure \ref{fig:densenet-cifar100}, reveal compelling insights. Both AdamD and AdamW consistently achieve 100\% training accuracy, whereas Adam falls short in this regard. From the training loss plots, we observe that the convergence speed of AdamD falls between that of AdamW and Adam. In most instances, AdamD achieves nearly the same level of generalization as AdamW. Moreover, the generalization capacity of Adam is notably inferior to that of the other two algorithms. This observation underscores the necessity of weight decoupling when solving the quadratically regularized problem defined in \eqref{Prob_Ori}.

To verify the results in Lemma \ref{Le_xk_yk_close}, we also present a plot of $\|x_k+\sigma m_k\|$ as shown in Figure \ref{fig:residual}. When $\theta_k$ adheres to a decay schedule described by $\mathcal{O}\left(k^{-\gamma}\right)$, \eqref{xk_yk_est} and basic calculus imply that $\|\sigma x_k+m_k\|$ exhibits an asymptotic behavior of $\mathcal{O}\left(k^{-\gamma}\right)$. The results in Figure \ref{fig:residual} are consistent with our theoretical analysis that $\{\norm{m_k+\sigma x_k}\}$ converges to 0, or equivalently $\{\norm{x_k-y_k}\}$ converges to 0. Notably, larger values of $\gamma$ correspond to a more rapid decline in $\|\sigma x_k+m_k\|$.

\subsubsection{LSTMs on language modeling}
In all our language modeling experiments, we consistently train our models for 200 epochs while employing a batch size of 128. Additionally, we adopt a stepsize reduction strategy that decreases the stepsize to 0.1 times its original value twice during training, specifically at the 75th and 150th epochs. These settings adhere to the commonly used experimental setup for training LSTMs, as demonstrated in previous works \cite{chen2021closing,zhuang2020adabelief}. This stepsize reduction strategy serves to accelerate the convergence of the optimization algorithm, simultaneously enhancing its generalization capacity. The weight decay parameter $\sigma$ {is fixed at} $1\times10^{-5}$ throughout these experiments. Our choice of hyperparameter settings aligns with those in Section \ref{subsubsec:CNN experiments}. The numerical results are displayed in Figure \ref{fig:lstm}.

From Figure \ref{fig:lstm}, we can observe that  both AdamD and Adam exhibit superior generalization capacity compared to AdamW. For 1- and 2-layer LSTM, AdamD exhibits similar generalization capacity compared to Adam. In the case of larger 3-layer LSTM models, AdamD outperforms Adam, achieving a test perplexity which is at least 2 units lower.

\subsection{Further Discussions on the AdamD Method}

\subsubsection{Asymptotic approximation to SGD sequence helps generalization}

As demonstrated in Lemma \ref{Le_xk_yk_close}, the term $\norm{ \sigma \xk +\mk}$  converges to $0$ as $k$ tends to infinity.  Then as discussed in Lemma \ref{Le_approx_dk}, the sequence $\{\yk\}$ (defined by $\yk := -\sigma^{-1} \mk$) approximately follows the update scheme \eqref{Eq_update_yk}, which asymptotically approximates a SGD method. 
Together with the fact that $\lim_{k\to \infty} \norm{\xk - \yk} = 0$, we can conclude  that the sequence $\{\xk\}$ in AdamD method is controlled by an SGD sequence $\{\yk\}$ as $k$ goes to infinity. Moreover, the interpolated process of $\{\yk\}$ is a perturbed solution of the differential inclusion \eqref{Eq_DI_SGD_yk}, i.e.,
\begin{equation}
	\label{Eq_Discussion_DI_SGD}
	\frac{\mathrm{d}y}{\mathrm{d}t} \in - (\D_f(y) + \sigma y). 
\end{equation}

On the other hand, in the early stage of the iterations of the AdamD method, the term $\norm{\sigma \xk + \mk}$ is large, and the ratio of $\theta_k$ and $\eta_k$ usually remains nearly unchanged. Then as illustrated in the discussion in Section 5, the sequence $\{(\xk, \mk, \vk)\}$ jointly tracks the trajectories of the differential inclusion 
\begin{equation}
	\label{Eq_Discussion_DI_Adam}
	\left(\frac{\mathrm{d}x}{\mathrm{d}t}, \frac{\mathrm{d}m}{\mathrm{d}t}, \frac{\mathrm{d}v}{\mathrm{d}t}\right)
	\in 
	-\left[\begin{matrix}
		(\ca{P}_+(v)  + \varepsilon )^{-\frac{1}{2}} \odot \left(   m + \sigma x \right)\\
		\tau_1 m - \tau_1 \D_f(x)\\
		\tau_2 v - \tau_2  \ca{U}(x)  \\
	\end{matrix}\right].
\end{equation}
Here $\ca{U}(x) := \frac{1}{N}\sum_{i = 1}^N \{d\odot d: d \in \D_{f_i}(x)\}$. Similar results are also exhibited in \cite{bianchi2022convergence,xiao2023adam}.   As the differential inclusion \eqref{Eq_Discussion_DI_Adam} imposes {preconditioners} to the update directions of $\{\xk\}$ based on the second-order moments of the stochastic subgradients, the sequence could quickly converge to a neighborhood of the stationary points.

These theoretical properties explain the fast convergence of the AdamD method in the early stage of the training and its lower generalization error than the Adam method with coupled weight decay. Based on the numerical experiments and our theoretical analysis, we believe the ability of asymptotically tracking an SGD sequence in AdamD helps to explain its superior generalization performance {over the} Adam method.

\subsubsection{Decoupled weight decay {is equivalent to} quadratic regularization}

It is conjectured in \cite{loshchilov2017decoupled} that the quadratic regularization terms contribute to the low generalization error in training neural networks. Finally, the authors in \cite{loshchilov2017decoupled} develop the AdamW method, showing that the weight decay is not equivalent to the quadratic regularization. As a result, the term $\sigma \xk$ in AdamW is not scaled by the preconditioner $(\sqrt{\vkp} + \varepsilon)^{-1}$. Therefore, the AdamW method does not have a clear objective function and lacks convergence guarantees in training nonsmooth neural networks. 

In our AdamD method, the objective function is exactly the $g(x)$ in \eqref{Prob_Ori}. Hence the weight decay parameter $\sigma$ is exactly the penalty parameter for the quadratic penalty term $\frac{\sigma}{2} \norm{x}^2$ in \eqref{Prob_Ori}. More importantly, we provide theoretical guarantees for the AdamD method in training nonsmooth neural networks. The stationarity of the iterates $\{\xk\}$ is characterized by $\D_f(\xk) + \sigma \xk$, hence has clearer meaning when compared with AdamW.

Furthermore, our numerical experiments demonstrate the superior performance of the AdamD method, illustrating that employing the quadratic regularization term in \eqref{Prob_Ori} does not undermine the generalization error.  Based on these results, we can conclude that, within our framework \eqref{Eq_Framework}, the weight decay can be interpreted as the quadratic regularization, which is different from the demonstrations in \cite{loshchilov2017decoupled} regarding AdamW.

\section{Conclusion}


In this paper, motivated by the AdamW method, we propose a novel framework \eqref{Eq_Framework} for Adam-family methods with decoupled weight decay. We prove that under mild assumptions with non-diminishing stepsizes $\{\eta_k\}$, any cluster point of $\{\xk\}$ is a $\D_g$-stationary point of \eqref{Prob_Ori}. Compared with the AdamW method, our proposed framework \eqref{Eq_Framework} enjoys convergence guarantees in training nonsmooth neural networks, and yields solutions that have clearer meanings. More importantly, we prove that the decoupled weight decay drives $\{\xk\}$ in the framework \eqref{Eq_Framework} to asymptotically approximate the SGD method. This fact provides {an intuitive}  understanding of the role of decoupled weight decay in Adam-family methods and explains the superior generalization performance of the Adam method with decoupled weight decay. 

As an application of our proposed framework \eqref{Eq_Framework}, we develop a novel Adam-family method named Adam with decoupled weight decay (AdamD), and prove its convergence properties under mild conditions. Numerical experiments on image classification and language modeling demonstrate the effectiveness of our proposed method. To conclude, we believe that our work {has enriched the theoretical understanding of weight decay and 
	explained its} practical utility in the field of deep learning applications.
	
	\bibliographystyle{plain}
	\bibliography{ref}

\begin{thebibliography}{10}

\bibitem{barakat2021convergence}
Anas Barakat and Pascal Bianchi.
\newblock Convergence and dynamical behavior of the {ADAM} algorithm for
  nonconvex stochastic optimization.
\newblock {\em SIAM Journal on Optimization}, 31(1):244--274, 2021.

\bibitem{benaim2006dynamics}
Michel Bena{\"\i}m.
\newblock Dynamics of stochastic approximation algorithms.
\newblock In {\em Seminaire de probabilites XXXIII}, pages 1--68. Springer,
  2006.

\bibitem{benaim2005stochastic}
Michel Bena{\"\i}m, Josef Hofbauer, and Sylvain Sorin.
\newblock Stochastic approximations and differential inclusions.
\newblock {\em SIAM Journal on Control and Optimization}, 44(1):328--348, 2005.

\bibitem{bianchi2022convergence}
Pascal Bianchi, Walid Hachem, and Sholom Schechtman.
\newblock Convergence of constant step stochastic gradient descent for
  non-smooth non-convex functions.
\newblock {\em Set-Valued and Variational Analysis}, pages 1--31, 2022.

\bibitem{bianchi2021closed}
Pascal Bianchi and Rodolfo Rios-Zertuche.
\newblock A closed-measure approach to stochastic approximation.
\newblock {\em arXiv preprint arXiv:2112.05482}, 2021.

\bibitem{bolte2007clarke}
J{\'e}r{\^o}me Bolte, Aris Daniilidis, Adrian Lewis, and Masahiro Shiota.
\newblock Clarke subgradients of stratifiable functions.
\newblock {\em SIAM Journal on Optimization}, 18(2):556--572, 2007.

\bibitem{bolte2022subgradient}
J{\'e}r{\^o}me Bolte, Tam Le, and Edouard Pauwels.
\newblock Subgradient sampling for nonsmooth nonconvex minimization.
\newblock {\em arXiv preprint arXiv:2202.13744}, 2022.

\bibitem{bolte2021nonsmooth}
J{\'e}r{\^o}me Bolte, Tam Le, Edouard Pauwels, and Tony Silveti-Falls.
\newblock Nonsmooth implicit differentiation for machine-learning and
  optimization.
\newblock {\em Advances in Neural Information Processing Systems}, 34, 2021.

\bibitem{bolte2020mathematical}
J{\'e}r{\^o}me Bolte and Edouard Pauwels.
\newblock A mathematical model for automatic differentiation in machine
  learning.
\newblock {\em Advances in Neural Information Processing Systems},
  33:10809--10819, 2020.

\bibitem{bolte2021conservative}
J{\'e}r{\^o}me Bolte and Edouard Pauwels.
\newblock Conservative set valued fields, automatic differentiation, stochastic
  gradient methods and deep learning.
\newblock {\em Mathematical Programming}, 188(1):19--51, 2021.

\bibitem{bolte2022long}
J{\'e}r{\^o}me Bolte, Edouard Pauwels, and Rodolfo Rios-Zertuche.
\newblock Long term dynamics of the subgradient method for {Lipschitz} path
  differentiable functions.
\newblock {\em Journal of the European Mathematical Society}, 2022.

\bibitem{bolte2022differentiating}
J{\'e}r{\^o}me Bolte, Edouard Pauwels, and Antonio~Jos{\'e} Silveti-Falls.
\newblock Differentiating nonsmooth solutions to parametric monotone inclusion
  problems.
\newblock {\em arXiv preprint arXiv:2212.07844}, 2022.

\bibitem{borkar2009stochastic}
Vivek~S Borkar.
\newblock {\em Stochastic approximation: a dynamical systems viewpoint},
  volume~48.
\newblock Springer, 2009.

\bibitem{bos1996using}
Siegfried Bos and E~Chug.
\newblock Using weight decay to optimize the generalization ability of a
  perceptron.
\newblock In {\em Proceedings of International Conference on Neural Networks
  (ICNN'96)}, volume~1, pages 241--246. IEEE, 1996.

\bibitem{castera2021inertial}
Camille Castera, J{\'e}r{\^o}me Bolte, C{\'e}dric F{\'e}votte, and Edouard
  Pauwels.
\newblock An inertial {Newton} algorithm for deep learning.
\newblock {\em The Journal of Machine Learning Research}, 22(1):5977--6007,
  2021.

\bibitem{chen2021closing}
Jinghui Chen, Dongruo Zhou, Yiqi Tang, Ziyan Yang, Yuan Cao, and Quanquan Gu.
\newblock Closing the generalization gap of adaptive gradient methods in
  training deep neural networks.
\newblock In {\em Proceedings of the Twenty-Ninth International Conference on
  International Joint Conferences on Artificial Intelligence}, pages
  3267--3275, 2021.

\bibitem{chen2023symbolic}
Xiangning Chen, Chen Liang, Da~Huang, Esteban Real, Kaiyuan Wang, Yao Liu, Hieu
  Pham, Xuanyi Dong, Thang Luong, Cho-Jui Hsieh, et~al.
\newblock Symbolic discovery of optimization algorithms.
\newblock {\em arXiv preprint arXiv:2302.06675}, 2023.

\bibitem{clarke1990optimization}
Frank~H Clarke.
\newblock {\em Optimization and nonsmooth analysis}, volume~5.
\newblock SIAM, 1990.

\bibitem{daniilidis2020pathological}
Aris Daniilidis and Dmitriy Drusvyatskiy.
\newblock Pathological subgradient dynamics.
\newblock {\em SIAM Journal on Optimization}, 30(2):1327--1338, 2020.

\bibitem{davis2020stochastic}
Damek Davis, Dmitriy Drusvyatskiy, Sham Kakade, and Jason~D Lee.
\newblock Stochastic subgradient method converges on tame functions.
\newblock {\em Foundations of Computational Mathematics}, 20(1):119--154, 2020.

\bibitem{ding2023nonconvex}
Kuangyu Ding, Jingyang Li, and Kim-Chuan Toh.
\newblock Nonconvex stochastic {Bregman} proximal gradient method with
  application to deep learning.
\newblock {\em arXiv preprint arXiv:2306.14522}, 2023.

\bibitem{duchi2018stochastic}
John~C Duchi and Feng Ruan.
\newblock Stochastic methods for composite and weakly convex optimization
  problems.
\newblock {\em SIAM Journal on Optimization}, 28(4):3229--3259, 2018.

\bibitem{guo2021novel}
Zhishuai Guo, Yi~Xu, Wotao Yin, Rong Jin, and Tianbao Yang.
\newblock A novel convergence analysis for algorithms of the {Adam} family.
\newblock {\em NeurIPS OPT Workshop}, 2021.

\bibitem{gurbuzbalaban2022stochastic}
Mert G{\"u}rb{\"u}zbalaban, Andrzej Ruszczy{\'n}ski, and Landi Zhu.
\newblock A stochastic subgradient method for distributionally robust
  non-convex and non-smooth learning.
\newblock {\em Journal of Optimization Theory and Applications},
  194(3):1014--1041, 2022.

\bibitem{he2016deep}
Kaiming He, Xiangyu Zhang, Shaoqing Ren, and Jian Sun.
\newblock Deep residual learning for image recognition.
\newblock In {\em Proceedings of the IEEE conference on computer vision and
  pattern recognition}, pages 770--778, 2016.

\bibitem{hochreiter1997long}
Sepp Hochreiter and J{\"u}rgen Schmidhuber.
\newblock Long short-term memory.
\newblock {\em Neural computation}, 9(8):1735--1780, 1997.

\bibitem{hu2022constraint}
Xiaoyin Hu, Nachuan Xiao, Xin Liu, and Kim-Chuan Toh.
\newblock A constraint dissolving approach for nonsmooth optimization over the
  {S}tiefel manifold.
\newblock {\em arXiv preprint arXiv:2205.10500}, 2022.

\bibitem{hu2022improved}
Xiaoyin Hu, Nachuan Xiao, Xin Liu, and Kim-Chuan Toh.
\newblock An improved unconstrained approach for bilevel optimization.
\newblock {\em arXiv preprint arXiv:2208.00732}, 2022.

\bibitem{huang2018condensenet}
Gao Huang, Shichen Liu, Laurens Van~der Maaten, and Kilian~Q Weinberger.
\newblock Condensenet: An efficient densenet using learned group convolutions.
\newblock In {\em Proceedings of the IEEE conference on computer vision and
  pattern recognition}, pages 2752--2761, 2018.

\bibitem{keskar2017improving}
Nitish~Shirish Keskar and Richard Socher.
\newblock Improving generalization performance by switching from {Adam} to
  {SGD}.
\newblock {\em arXiv preprint arXiv:1712.07628}, 2017.

\bibitem{kingma2014adam}
Diederik~P Kingma and Jimmy Ba.
\newblock Adam: A method for stochastic optimization.
\newblock {\em In Proceedings of the 3rd International Conference for Learning
  Representations}, 2015.

\bibitem{krizhevsky2009learning}
Alex Krizhevsky, Geoffrey Hinton, et~al.
\newblock Learning multiple layers of features from tiny images.
\newblock 2009.

\bibitem{krogh1991simple}
Anders Krogh and John Hertz.
\newblock A simple weight decay can improve generalization.
\newblock {\em Advances in neural information processing systems}, 4, 1991.

\bibitem{le2023nonsmooth}
Tam Le.
\newblock Nonsmooth nonconvex stochastic heavy ball.
\newblock {\em arXiv preprint arXiv:2304.13328}, 2023.

\bibitem{liu2019variance}
Liyuan Liu, Haoming Jiang, Pengcheng He, Weizhu Chen, Xiaodong Liu, Jianfeng
  Gao, and Jiawei Han.
\newblock On the variance of the adaptive learning rate and beyond.
\newblock {\em arXiv preprint arXiv:1908.03265}, 2019.

\bibitem{loshchilov2017decoupled}
Ilya Loshchilov and Frank Hutter.
\newblock Decoupled weight decay regularization.
\newblock {\em arXiv preprint arXiv:1711.05101}, 2017.

\bibitem{luo2019adaptive}
Liangchen Luo, Yuanhao Xiong, Yan Liu, and Xu~Sun.
\newblock Adaptive gradient methods with dynamic bound of learning rate.
\newblock {\em arXiv preprint arXiv:1902.09843}, 2019.

\bibitem{marcus1993building}
Mitchell Marcus, Beatrice Santorini, and Mary~Ann Marcinkiewicz.
\newblock Building a large annotated corpus of english: The penn treebank.
\newblock 1993.

\bibitem{reddi2019convergence}
Sashank~J Reddi, Satyen Kale, and Sanjiv Kumar.
\newblock On the convergence of {Adam} and beyond.
\newblock {\em In 6th International Conference on Learning Representations
  (ICLR)}, 2018.

\bibitem{ruszczynski2020convergence}
Andrzej Ruszczy{\'n}ski.
\newblock Convergence of a stochastic subgradient method with averaging for
  nonsmooth nonconvex constrained optimization.
\newblock {\em Optimization Letters}, 14(7):1615--1625, 2020.

\bibitem{ruszczynski2021stochastic}
Andrzej Ruszczynski.
\newblock A stochastic subgradient method for nonsmooth nonconvex multilevel
  composition optimization.
\newblock {\em SIAM Journal on Control and Optimization}, 59(3):2301--2320,
  2021.

\bibitem{shi2021rmsprop}
Naichen Shi, Dawei Li, Mingyi Hong, and Ruoyu Sun.
\newblock Rmsprop converges with proper hyperparameter.
\newblock In {\em International Conference on Learning Representation}, 2021.

\bibitem{van1996geometric}
Lou Van~den Dries and Chris Miller.
\newblock Geometric categories and o-minimal structures.
\newblock {\em Duke Mathematical Journal}, 84(2):497--540, 1996.

\bibitem{wang2022provable}
Bohan Wang, Yushun Zhang, Huishuai Zhang, Qi~Meng, Zhi-Ming Ma, Tie-Yan Liu,
  and Wei Chen.
\newblock Provable adaptivity in adam.
\newblock {\em arXiv preprint arXiv:2208.09900}, 2022.

\bibitem{xiao2023adam}
Nachuan Xiao, Xiaoyin Hu, Xin Liu, and Kim-Chuan Toh.
\newblock Adam-family methods for nonsmooth optimization with convergence
  guarantees.
\newblock {\em arXiv preprint arXiv:2305.03938}, 2023.

\bibitem{xiao2023convergence}
Nachuan Xiao, Xiaoyin Hu, and Kim-Chuan Toh.
\newblock Convergence guarantees for stochastic subgradient methods in
  nonsmooth nonconvex optimization.
\newblock {\em arXiv preprint arXiv:2307.10053}, 2023.

\bibitem{xiao2023dissolving}
Nachuan Xiao, Xin Liu, and Kim-Chuan Toh.
\newblock Dissolving constraints for {R}iemannian optimization.
\newblock {\em Mathematics of Operations Research}, 2023.

\bibitem{zaheer2018adaptive}
Manzil Zaheer, Sashank Reddi, Devendra Sachan, Satyen Kale, and Sanjiv Kumar.
\newblock Adaptive methods for nonconvex optimization.
\newblock {\em {A}dvances in {N}eural {I}nformation {P}rocessing {S}ystems},
  31, 2018.

\bibitem{zhang2022adam}
Yushun Zhang, Congliang Chen, Naichen Shi, Ruoyu Sun, and Zhi-Quan Luo.
\newblock Adam can converge without any modification on update rules.
\newblock {\em Advances in Neural Information Processing Systems},
  35:28386--28399, 2022.

\bibitem{zhou2020towards}
Pan Zhou, Jiashi Feng, Chao Ma, Caiming Xiong, Steven Chu~Hong Hoi, et~al.
\newblock Towards theoretically understanding why {SGD} generalizes better than
  {Adam} in deep learning.
\newblock {\em Advances in Neural Information Processing Systems},
  33:21285--21296, 2020.

\bibitem{zhou2022towards}
Pan Zhou, Xingyu Xie, and YAN Shuicheng.
\newblock Towards understanding convergence and generalization of adamw.
\newblock 2022.

\bibitem{zhuang2020adabelief}
Juntang Zhuang, Tommy Tang, Yifan Ding, Sekhar~C Tatikonda, Nicha Dvornek,
  Xenophon Papademetris, and James Duncan.
\newblock Adabelief optimizer: Adapting stepsizes by the belief in observed
  gradients.
\newblock {\em Advances in neural information processing systems},
  33:18795--18806, 2020.

\bibitem{zou2019sufficient}
Fangyu Zou, Li~Shen, Zequn Jie, Weizhong Zhang, and Wei Liu.
\newblock A sufficient condition for convergences of {Adam} and {RMSProp}.
\newblock In {\em Proceedings of the IEEE/CVF conference on computer vision and
  pattern recognition}, pages 11127--11135, 2019.

\end{thebibliography}

\end{document}